\documentclass[12pt,reqno]{amsart}%%%%%%%%%%%%{article}
\usepackage {amssymb}
\usepackage {amsmath}
\usepackage {bbm}
\usepackage{amsthm}
\usepackage{mathtools}
\usepackage{graphicx}
\usepackage {amscd}
\usepackage {epic}
\usepackage[alphabetic]{amsrefs}
\usepackage[colorlinks, citecolor=blue]{hyperref}
\usepackage{enumerate}
\usepackage{tikz}
\usepackage{tikz-cd}
\usepackage{verbatim,color,geometry}
\usepackage[all]{xy}
\usepackage{enumitem}
\newcommand{\Aut}[0]{\operatorname{Aut}}

\newcommand{\kkX}{{\mathfrak{X} }}

\newcommand{\Z}{{\mathbb{Z} }}

\newcommand{\cB}{{\mathcal{B}}}
\newcommand{\cR}{{\mathcal{R}}}
\newcommand{\cX}{{\mathcal{X}}}
\newcommand{\cW}{{\mathcal{W}}}
\newcommand{\cY}{{\mathcal{Y}}}

\newcommand{\mult}{{\rm mult}}
\newcommand{\Supp}{{\rm Supp}}

\newcommand{\ord}{{\rm ord}}
\newcommand{\gr}{{\rm gr}}

\newtheorem{thm}{Theorem}[section]

\newtheorem{lem}[thm]{Lemma}
\newtheorem{cor}[thm]{Corollary}

\newtheorem{prop}[thm]{Proposition}

\theoremstyle{definition}
\newtheorem{defn}[thm]{Definition}
\newtheorem{condition}[thm]{Condition}

\newtheorem{example}[thm]{Example}

\newtheorem{rem}[thm]{Remark}

\newtheorem{defn-thm}[thm]{Definition--Theorem}  %!!!!!!!!!!!!!!!!!!!!!!!!
\newtheorem{defn-prop}[thm]{Definition--Proposition}  %!!!!!!!!!!!!!!!!!!!!!!!!
\newtheorem{defn-lem}[thm]{Definition--Lemma}  %!!!!!!!!!!!!!!!!!!!!!!!!
\theoremstyle{remark}

\input{diagrams.sty}
\newcommand{\Fut}{{\rm Fut}}
\newcommand{\Proj}{{\rm Proj}}
\newcommand{\Val}{{\rm Val}}
\newcommand{\vol}{{\rm vol}}
\newcommand{\hvol}{\widehat{\rm vol}}
\newcommand{\cD}{{\mathcal{D}}}
\newcommand{\cG}{{\mathcal{G}}}
\newcommand{\cO}{{\mathcal{O}}}
\newcommand{\cV}{{\mathcal{V}}}
\newcommand{\fa}{{\mathfrak{a}}}
\newcommand{\lct}{{\rm lct}}

\newcommand{\bQ}{{\mathbb{Q}}}
\newcommand{\Q}{{\mathbb{Q}}}
\newcommand{\N}{{\mathbb{N}}}
\newcommand{\fb}{{\mathfrak{b}}}

\newcommand{\Spec}{{\rm Spec}}

\newcommand{\e}{\varepsilon}

\newcommand{\cF}{\mathcal{F}}

\usepackage[normalem]{ulem}

\newcommand{\bk}{k}
\newcommand{\co}{\colon}

\def\gitq{/\hspace{-0.1cm}/}

\renewcommand{\ss}{{\rm ss}}
\renewcommand{\tilde}{\widetilde}

\newcommand{\bA}{\mathbb{A}}
\newcommand{\bG}{\mathbb{G}}
\newcommand{\bP}{\mathbb{P}}
\newcommand{\bZ}{\mathbb{Z}}

\newcommand{\oh}{\mathcal{O}}
\newcommand{\cE}{\mathcal{E}}
\newcommand{\cI}{\mathcal{I}}
\newcommand{\cL}{\mathcal{L}}
\newcommand{\cP}{\mathcal{P}}
\newcommand{\cS}{\mathcal{S}}

\newcommand{\uProj}{\underline{{\rm Proj}}}

\DeclareMathOperator{\ST}{ST}
\DeclareMathOperator{\colim}{colim}
\DeclareMathOperator{\PGL}{PGL}
\DeclareMathOperator{\Isom}{Isom}
\newcommand{\ilim}{\varprojlim}

\usepackage[normalem]{ulem}

\begin{document}
\title{Reductivity of the automorphism group of K-polystable Fano varieties}
\date{\today}

\author {Jarod Alper}
\address   {Department of Mathematics, University of Washington, Box 354350, Seattle, WA
98195-4350, USA}
\email     {jarod@uw.edu}

\author{Harold Blum}
\address{Department of Mathematics, University of Utah, Salt Lake City, UT 48112, USA}
\email{blum@math.utah.edu}

\author {Daniel Halpern-Leistner}
\address{Malott Hall, Mathematics Dept, Cornell University, Ithaca, NY 14853, USA}
\email     {daniel.hl@cornell.edu}

\author {Chenyang Xu}
\address   {MIT,       Cambridge 02139, USA}
\email     {cyxu@math.mit.edu}
\address   {Beijing International Center for Mathematical Research,       Beijing 100871, China}
\email     {cyxu@math.pku.edu.cn}

\begin{abstract}{
We prove that  K-polystable log Fano pairs have reductive automorphism groups. 
In fact, we deduce this statement by establishing more general results concerning the S-completeness and $\Theta$-reductivity of the moduli of K-semistable 
log Fano pairs. 
Assuming the conjecture that K-semistability is an open condition, we prove that the Artin stack parametrizing K-semistable Fano varieties admits a separated good moduli space.
%Also, need the stack to contain the closures of geometric points
}
\end{abstract}

\maketitle{}
\setcounter{tocdepth}{1}
\tableofcontents

{\let\thefootnote\relax\footnotetext{
JA was partially supported by NSF grant DMS-1801976. 
HB was partially supported by NSF grant DMS-1803102.  
DHL was partially supported by NSF grant DMS-1762669. 
CX was partially supported by a Chern Professorship of the MSRI (NSF No. DMS-1440140), by NSFC grant No. 11425101 (2015-2018) and by
NSF grant DMS-1901849.
}
\marginpar{}
}

\emph{Throughout, we work over an algebraically closed  field $k$ of characteristic 0.}

\section{Introduction}\label{s-intro}

The construction of moduli spaces parametrizing K-semistable and K-polystable Fano varieties is a profound goal in the study of Fano varieties. 
The \emph{K-moduli Conjecture} predicts that the moduli functor 
$\kkX^{\rm Kss}_{n,V}$ of K-semistable $\mathbb{Q}$-Fano varieties of dimension $n$ and volume $V$, which sends a $k$-scheme $S$ to
 \[
{\kkX^{\rm Kss}_{n,V}} (S) = 
 \left\{
  \begin{tabular}{c}
\mbox{Flat proper families $X\to S$, whose geometric fibers are }\\
\mbox{K-semistable $\Q$-Fano varieties of dimension $n$ and  }\\
\mbox{volume $V$, satisfying Koll\'ar's condition (see \cite[\S 1]{BX18})}
  \end{tabular}
\right\},
\]
is represented by a finite type Artin stack $\kkX^{\rm Kss}_{n,V}$  and it admits a projective good moduli space $\kkX^{\rm Kss}_{n,V}\to X^{\rm Kps}_{n,V}$ (see Definition \ref{d-goodmoduli}), whose closed points precisely parameterize $n$-dimensional K-polystable $\Q$-Fano varieties of volume $V$. 
The ingredients needed in the construction can be translated into deep properties of such Fano varieties. See \cite[Introduction]{BX18} for a more detailed discussion of the prior state of the art.

\subsection{Main theorems}

In this paper, we show that if the moduli functor $\kkX^{\rm Kss}_{n,V}$ is represented by an Artin stack, then it admits a separated good moduli space (see Step (III) in \cite[Introduction]{BX18}). A prototype of the good moduli space of a stack is given by the morphism $[X^{\rm ss}/G]\to X /\!\!/G$ to the geometric invariant theory (GIT) quotient of a polarized projective variety $(X,L)$ by a reductive group $G$. However, for the question of K-stability of Fano varieties, it is not clear how to realize it as a GIT question: on the one hand, we know there are K-polystable Fano varieties which are not asymptotically Chow semistable (see e.g. \cite{OSY12, LLSW17}); on the other hand, the more natural CM line bundle is not positive on the Hilbert scheme (see \cite{FR06}).

Roughly speaking, for moduli problems which are not known to be global GIT quotients, however, we still aim to find a quotient  space, such that the quotient morphism behaves as well as the GIT quotient morphism $[X^{\rm ss}/G]\to X /\!\!/G$ from many perspectives (see Definition \ref{d-goodmoduli}).  
In this note, we adapt the general framework developed in \cite{AHH18} to the case of K-semistable $\mathbb{Q}$-Fano varieties.
\begin{thm}\label{t-main}
The functor $\kkX^{\rm Kss}_{n,V}$ satisfies the valuative criterion for  $S$-completeness (see Definition \ref{d-scomplete}) and $\Theta$-reductivity (see Definition \ref{d-thetared}) with respect to essentially of finite type DVRs.
\end{thm}

For an Artin stack of finite type with affine diagonal over a field of characteristic $0$, \cite[Theorem A]{AHH18} states that the conditions of S-completeness and $\Theta$-reductivity are equivalent to the existence of a separated good moduli space. An immediate corollary is that
\begin{cor}\label{c-goodmoduli}
Let $\kkX\subset \kkX^{\rm Kss}_{n,V}$ be a subfunctor representable by an Artin stack of finite type, such that if $x\in \kkX$ then $\overline{\{x\}}\subset \kkX$. Then $\kkX$ admits a separated good moduli space.
\end{cor}

The stack $\kkX_{n,V}$ is an Artin stack with affine diagonal, and it is known that the semistable locus is bounded (cf. \cite{Jia17}), so it remains to show that $\kkX_{n,V}^{\rm Kss} \subset \kkX_{n,V}$ is an open substack (see \cite[Step II]{BX18}). This question was settled shortly after this paper was first released (see Remark \ref{r-postscript}). For smoothable K-semistable Fano varieties,  the existence of the good moduli space as well as its properness were settled in \cite{LWX14}. %Of course, we expect $\kkX^{\rm Kss}_{n,V}$ itself is representable by an Artin stack of finite type.

%At the moment, the openness of semistability is still a conjecture.

In fact, we prove $S$-completeness and $\Theta$-reductivity of the moduli functor parameterizing families of K-semistable log Fano pairs. Since $S$-completeness implies the reductivity of the automorphism group of any polystable point, we can conclude:   

\begin{thm}\label{t-aut}
If $(X,D)$ is a K-polystable log Fano pair, then ${\rm Aut}(X,D)$ is reductive.
\end{thm}

This theorem has a long history: it is a classical result for K\"ahler-Einstein Fano manifolds in \cite{Mat57} (and even holds in the more general  case of polarized manifolds with  constant scalar curvature metrics). 
For log Fano pairs with a weak conical K\"ahler-Einstein metric, this is a much harder result and it is a key step in the proofs of the Yau-Tian-Donaldson Conjecture for smooth Fano manifolds (see e.g. \cites{CDS, Tia15, BBEGZ11}).
 Our method is purely algebro-geometric. In \cite{BX18},  it was shown that if $(X,D)$ is K-stable, then ${\rm Aut}(X,D)$ is finite. That paper also establishes a key ingredient in the proof of Theorem \ref{t-aut}, the Finite Generation Condition \ref{a-finitegenerate}. We also note that when $X$ is only K-semistable, then ${\rm Aut}(X)$ can be non-reductive (see \cite{CS18}*{Example 1.4}). 

\subsection{Sketch of the proof}

We sketch the main ideas in the proof of Theorem \ref{t-main}. 
The conditions of $S$-completeness and $\Theta$-reductivity of $\kkX^{\rm Kss}_{n,V}$ both involve extending a  family of K-semistable $\mathbb{Q}$-Fano varieties over the complement of a closed point in a certain regular surface to a family over the entire surface. We first show that the pushforward sheaves of $m$-th relative anti-pluri-canonical line bundles extend, then we prove that the direct sum of these sheaves is finitely generated.  After taking $\Proj$ of this algebra, we argue that the central fiber is  a K-semistable $\mathbb{Q}$-Fano variety, which gives the desired extension of the family of $K$-semistable $\mathbb{Q}$-Fano varieties. Of course, such finite generation results  are highly nontrivial. Fortunately, for families of K-semistable Fano varieties, the finite generation needed for S-completeness was essentially settled in \cite{BX18} and the case for $\Theta$-reductivity is proved in Section \ref{s-theta}, closely following similar arguments in \cite{LWX18}. This general strategy could  conceivably be applied to general K-semistable polarized varieties;  however, the corresponding finite generation statements  (see Conditions \ref{a-finitegenerate} and \ref{a-finitegenerate2}) appear to be very challenging.

\medskip

We now explain in more detail the proof of S-completeness.  
We say any two K-semistable $\bQ$-Fano varieties lie in the same S-equivalence class if they degenerate to a common K-semistable $\bQ$-Fano variety via special test configurations (see e.g.\cite{BX18}*{Def. 2.6}).
The first extensive study of the geometry of K-semistable $\mathbb{Q}$-Fano varieties belonging to the same S-equivalence class was completed in \cite{LWX18}. In particular, it was shown that there is a unique object, namely a K-polystable $\mathbb{Q}$-Fano variety, in each S-equivalence class.

Then in \cite{BX18}, the study of families of K-semistable Fano varieties is extended from test configurations to families over a curve.  
Namely, given two $\mathbb{Q}$-Gorenstein families of K-semistable $\Q$-Fano varieties $f\colon X\to C$ and $f'\colon X' \to C$ over the germ of a pointed smooth curve $(C={\rm Spec}(R),0)$ and an isomorphism $X\times_C (C \setminus 0) \cong X'\times_C (C \setminus 0)$, \cite{BX18} established that $X_0$ and $X_0'$ are always $S$-equivalent.  The argument for this fact can be divided into two parts: (1) one constructs filtrations $\cF$ and $\cF'$ of $V:=\bigoplus_m V_m=\bigoplus_m H^0(X_0,-mrK_{X_0})$ and $V'=\bigoplus_m V'_m = \bigoplus_m H^0(X'_0,-mrK_{X'_0} )$ for some fixed sufficiently divisible $r$ such that ${\rm gr}_\cF(V)=\bigoplus_m{\gr}_\cF(V_m)$ is isomorphic to ${\rm gr}_{\cF'}(V')=\bigoplus_m{\gr}_{\cF'}(V'_m)$, and (2) one shows that the above graded rings are indeed finitely generated and moreover that their ${\rm Proj}$ give a common K-semistable degeneration of $X_0$ and $X_0'$.

Meanwhile, the property of S-completeness was introduced in \cite{AHH18} as part of a general criterion for the existence of good moduli space (see Theorem \ref{T:existence}).
 The first key observation in this paper is that the construction of the filtration in \cite{BX18} indeed can be put into this framework of S-completeness. More precisely, in the current note, we verify that for each fixed $m$, in the above construction from \cite{BX18}, the $m$-th graded module, ${\gr}_\cF(V_m)\cong {\rm gr}_{\cF'}(V'_m)$ is  precisely the fiber over $0$ of the pushforward along $ \overline{\ST}_R \setminus 0\subset \overline{\ST}_R $  (where $\overline{\ST}_R$ is a local model of the quotient $[\bA^2/\bG_m]$ with weights $1$ and $-1$---see \eqref{E:ST} for the precise definition) of  the locally free sheaf over $ \overline{\ST}_R \setminus 0$ obtained by gluing $\cV_m=f_*(-mrK_{X/C})$ and $\cV'_m=f'_*(-mrK_{X'/C})$. Indeed, we show that the graded module in \cite{BX18} is the same, up to a grading shift, as the one naturally arising from the module over $\overline{\ST}_R$. Hence by taking the direct sum over all $m$, we produce a graded algebra over $\overline{\ST}_R$, which is finitely generated exactly by the finite generation results proved in \cite{BX18}. Finally, by taking the ${\rm Proj}$, we construct the extended family of K-semistable $\mathbb{Q}$-Fano varieties over $\overline{\ST}_R$.

In some sense,  the $S$-completeness criterion in  \cite{AHH18}  provides a conceptual framework for enhancing the `pointwise'  results in  \cite{LWX18, BX18} to results over families. Remarkably, this even yields new results for a single Fano variety, e.g. Theorem \ref{t-aut}.

\medskip

To prove the $\Theta$-reductivity (see Definition \ref{d-thetared}), we need to show that, given a family of K-semistable $\mathbb{Q}$-Fano varieties $f\colon X\to C$ over the germ of a pointed curve $(C={\rm Spec}(R),0)$,  any family of test configurations for $X\times_C (C \setminus 0)$ over $C \setminus 0$ with K-semistable central fibers can be extended to a family of test configurations for $X$ over $C$ with K-semistable central fibers. When $X/C$ itself is a test configuration, the proof is contained in \cite{LWX18}. To establish the $\Theta$-reductivity, we need to generalize the argument in \cite{LWX18} from the base curve being $\Theta=[\mathbb{A}^1/\mathbb{G}_m]$ to a more general base curve $C$. % from the case that the base is $\mathbb{A}^1$ with a $\mathbb{G}_m$-action to a general base curve. 
Nevertheless, the techniques are similar. 

\medskip

\begin{rem}[Postscript] \label{r-postscript}
After the first version of the current paper was written, there were two related developments. First, it was proved in \cite{BLX19} and \cite{Xu20} that,  for a family of log Fano pairs, the locus where the fibers are K-semistable is open. This together with \cite{Jia17} implies the functor $\kkX^{\rm Kss}_{n,V}$  is represented by an Artin stack of finite type. Therefore, we can apply Theorem \ref{c-goodmoduli} to $\kkX^{\rm Kss}_{n,V}$ itself and conclude it admits a good moduli space. 
Second, the moduli functors of log Fano pairs over a general base has been appropriately defined in \cite{Kol19b}, which also can be shown to be represented by an Artin stack. The results in this paper then confirm this Artin stack also has a good moduli space. For a detailed account, see \cite{XZ19}*{Sec. 2.6}. 
\end{rem}

%\subsection{The organization of the paper:} In Section \ref{s-gms}, we discuss some preliminaries on good moduli spaces and K-semistable log Fano pairs. In Section \ref{s-comparison}, we carry out the explicit calculations for sheaves,which will be applied to the direct image of $m$-th power of the polarization of a family of polarized varieties. Then in Section \ref{} , we show that in the case of families of K-semistable $\mathbb{Q}$-Fano varieties, if we put all $m$ together, the algebra is finitely generated, which then verify S-completeness and $\Theta$-reductivity for the functor of K-semistable $\mathbb{Q}$-Fano varieties.

 \bigskip

%One way to attack the equivalence between uniform K-stability and K-stability is using the $\delta$-invariant.

\noindent{\bf Acknowledgement: }  JA and CX thank Xiaowei Wang, and CX thanks Jun Yu for helpful conversations.  We thank the referees for suggestions on revising the paper. Much of the work on this paper was completed
while the authors enjoyed the hospitality of the MSRI, which is gratefully acknowledged.

\section{Preliminaries}\label{s-gms}

\subsection{Good moduli spaces}
In this section, we discuss some general facts about {\it good moduli spaces}. The following definition was introduced in \cite{Alp13}.

\begin{defn}[Good moduli space]\label{d-goodmoduli}
If $\cX$ is an Artin stack of finite type over $\bk$, a morphism $\phi \co \cX \to X$ to an algebraic space is called a {\it good moduli space} if (1) $\phi_*$ is exact on the category of coherent $\oh_{\cX}$-modules and (2) $\oh_X \to \phi_* \oh_{\cX}$ is an isomorphism.
\end{defn}

\begin{rem} We note that $X$ is unique as the map $\cX \to X$ is initial for maps to algebraic spaces \cite[Thm.~6.6]{Alp13} and $X$ is necessarily of finite type over $\bk$ \cite[Thm.~4.16(xi)]{Alp13}.  Moreover, two $\bk$-points of $\cX$ are identified in $X$ if and only if their closures intersect \cite[Thm.~4.16(iv)]{Alp13}.  In particular, there is a bijection between the closed $\bk$-points of $\cX$ (i.e. \emph{the polystable objects}) and the $\bk$-points of $X$.  

The canonical example arises from GIT: if $G$ is a reductive group acting on a closed $G$-invariant subscheme $X \subset \bP(V)$, where $V$ is a finite dimensional $G$-representation, then the morphism 
$$[X^{\ss}/G] \to X^{\ss}\gitq G := \Proj \bigoplus_m H^0(X, \oh_X(m))^G$$ to the GIT quotient is a good moduli space. 

However, the K-stability moduli problem does not have a known GIT interpretation. So to prove the moduli stack $\kkX^{\rm Kss}_{n,V}$ yields a good moduli space $X^{\rm Kps}_{n,V}$ is quite nontrivial.   
%For instance, it was shown in \cite{LWX18}, that the K-polystable $\mathbb{Q}$-Fano varieties is a \emph{polystable object}, i.e., it is closed in the stack of
\end{rem}

\subsubsection{S-completeness}

Let $R$ be a DVR over $k$ with fraction field $K$, residue field $\kappa$, and uniformizing parameter $\pi$.  We define the Artin stack
\begin{equation} \label{E:ST}
\overline{\ST}_R := [\Spec \big(R[s,t] / (st-\pi)\big)  / \bG_m],
\end{equation}
where $s$ and $t$ have weights $1$ and $-1$.  This can be viewed as a local model of the quotient $[\bA^2/\bG_m]$ where $\bA^2$ has coordinates $s$ and $t$ with weights $1$ and $-1$; indeed, $\overline{\ST}_R$ is the base change of the good moduli space $[\bA^2/\bG_m] \to \Spec(\bk[st])$ along $\Spec R \to  \Spec(\bk[st])$ defined by $st \mapsto \pi$.
We denote by $0 \in \overline{\ST}_R$ the unique closed point defined by the vanishing of $s$ and $t$.  Observe that $\overline{\ST}_R \setminus 0$ is the non-separated union $\Spec(R) \bigcup_{\Spec(K)} \Spec(R)$.  

Denote $\Theta_{\kappa} = [\bA_{\kappa}^1 / \bG_m]$ as the quotient of the usual scaling action.   The following two cartesian diagrams yield a useful schematic picture of $\overline{\ST}_R$

\begin{equation} \label{E:ST-schematic}
\begin{split}
\xymatrix{
			&	\Spec(R)\ar@{^(->}[rd]^{s \neq 0}		&		& \Theta_{\kappa} \ar@{_(->}[ld]_{s = 0}\\
\Spec(K)\ar@{^(->}[ru]\ar@{^(->}[rd]		&		& \overline{\ST}_R 		&		& B_{\kappa}\bG_m \ar@{_(->}[lu]\ar@{_(->}[ld]	\\
			& \Spec(R) \ar@{^(->}[ru]_{t \neq 0}	&	& \Theta_{\kappa} \ar@{_(->}[lu]^{t = 0}
}
\end{split}
\end{equation}
where the maps to the left are open immersions and to the right are closed immersions.

\begin{defn}[$S$-completeness]\label{d-scomplete} A stack $\cX$ over $\bk$  is \emph{S-complete} if for any DVR $R$ and any diagram
\begin{equation} \label{E:S-complete}
\begin{split}
\xymatrix{
\overline{\ST}_R  \setminus 0 \ar[r] \ar[d]					& \cX  \\
\overline{\ST}_R \ar@{-->}[ur]		&
}
\end{split} \end{equation}
there exists a unique dotted arrow filling in the diagram.  

Moreover, if $R$ is a DVR, we say that $\cX$ \emph{satisfies the valuative criterion for S-completeness for $R$} if any diagram \eqref{E:S-complete} can be uniquely filled in.
\end{defn}

\begin{rem}\label{r-reductive} 
This definition was introduced for Artin stacks in \cite[\S 3.5]{AHH18}.  At the time this paper was written, it was not known if $\kkX^{\rm Kss}_{n,V}$ was an Artin stack, so we were careful not to assume this about $\cX$. This question has since been resolved (see Remark \ref{r-postscript}).
%Since it is still open whether the stack $\kkX^{\rm Kss}_{n,V}$ of K-semistable $\bQ$-Fanos is represented by an Artin stack,\HL{I think it is weird to say this is still open. I think we should change this remark to say that after this paper was completed, it was shown that $\kkX^{\rm Kss}$ is an Artin stack.} we have removed the condition that $\cX$ is Artin in the above definition.  We point out that if it were known that $\kkX_{n,V}^{\rm ss} \subset \kkX_{n,V}$ is an open substack, then it would follow that $\kkX_{n,V}^{\rm ss}$ is an Artin stack with affine diagonal since $\kkX_{n,V}$ is.\footnote{\CX{Should we keep this remark or maybe we can remove it?}}
\end{rem}

\begin{rem} \label{r-reductive}
 If $\cX$ is Deligne-Mumford, then $\cX$ is S-complete if and only if $\cX$ is separated (\cite[Prop.~3.44]{AHH18}).  
 If $\cX$ is an Artin stack with affine diagonal, then any lift is automatically unique  (\cite[Prop.~3.40]{AHH18}). 
\end{rem}

\begin{rem} \label{R:reductive}
 If $G$ is a linear algebraic group over $\bk$, then $BG$ is S-complete (equivalently S-complete with respect to essentially of finite type DVRs) if and only if $G$ is reductive (\cite[Prop.~3.45 and Rem.~3.46]{AHH18}).  Moreover, as S-completeness is preserved under closed substacks, it follows that every closed point (i.e. polystable object) in an Artin stack with affine diagonal, which is S-complete with respect to essentially of finite type DVRs, has \emph{reductive stabilizer}.
\end{rem}

\subsubsection{$\Theta$-reductivity}

We define $\Theta = [\bA^1 / \bG_m]$ with coordinate $x$ on $\bA^1$ having weight $-1$, and we set $\Theta_R = \Theta \times_{\bk} \Spec(R)$ for any DVR $R$.  We let $0 \in \Theta_R$ be the unique closed point defined by the vanishing of $x$ and the uniformizing parameter $\pi \in R$.  Observe that $\Theta_R \setminus 0 = \Theta_K \bigcup_{\Spec(K)} \Spec(R)$.  Analogous to \eqref{E:ST-schematic}, we have the two following cartesian diagrams
\begin{equation} \label{E:ThetaR-schematic}
\begin{split}
\xymatrix{
			&	\Spec(R)\ar@{^(->}[rd]^{x \neq 0}		&		& B_R \bG_m \ar@{_(->}[ld]_{x = 0}\\
\Spec(K)\ar@{^(->}[ru]\ar@{^(->}[rd]		&		& \Theta_R 		&		& B_{\kappa}\bG_m \ar@{_(->}[lu]\ar@{_(->}[ld]	\\
			& \Theta_K \ar@{^(->}[ru]_{\pi \neq 0}	&	& \Theta_{\kappa} \ar@{_(->}[lu]^{\pi = 0}
}
\end{split}
\end{equation}
where the maps to the left are open immersions and to the right are closed immersions.

\begin{defn}[{$\Theta$-reductivity}]\label{d-thetared}
A stack $\cX$  over $\bk$  is \emph{$\Theta$-reductive} if for any DVR $R$ and any  diagram
\begin{equation} \label{E:Theta-reductive}
\begin{split}
\xymatrix{
\Theta_R  \setminus 0 \ar[r] \ar[d]					& \cX  \\
\Theta_R \ar@{-->}[ur]		&
}
\end{split} \end{equation}
there exists a unique dotted arrow filling in the diagram.

Moreover, if $R$ is a DVR, we say that $\cX$ \emph{satisfies the valuative criterion for $\Theta$-reductivity for $R$} if any diagram \eqref{E:Theta-reductive} can be uniquely filled in.
\end{defn}

\begin{rem} This definition was introduced in \cite{Hal14}.  As with S-completeness, if $\cX$ is an Artin stack with affine diagonal, then any lift is automatically unique. 
 \end{rem}
 
 \subsubsection{The existence of good moduli spaces} The following criterion is established in \cite{AHH18}.

\begin{thm}\cite[Thm.~A]{AHH18} \label{T:existence}
Let $\cX$ be an Artin stack of finite type with affine diagonal over $\bk$.  Then $\cX$ admits a good moduli space $\cX \to X$ with $X$ separated if and only if $\cX$ is S-complete and $\Theta$-reductive.
\end{thm}

 \begin{rem} \label{R:existence}
 The following technical refinement of Theorem \ref{T:existence} will be useful to us as we are unable to verify the valuative criteria for S-completeness and $\Theta$-reductivity for \emph{every} DVR $R$ (see Definitions \ref{d-scomplete} and \ref{d-thetared}). 
To show the existence of a good moduli space $\cX \to X$ with $X$ separated, it suffices to verify the valuative criteria for S-completeness and $\Theta$-reductivity for DVRs $R$ essentially of finite type over $\bk$ (\cite[Rmk.~5.5]{AHH18}). %should replace second reference
Once this is established, it follows in fact (from applying the converse of Theorem \ref{T:existence}) that $\cX$ satisfies the valuative criteria for S-completeness and $\Theta$-reductivity for all DVRs $R$.
  \end{rem}

\begin{rem}[Comparing with an earlier criterion] 
In \cite{LWX14}, a variant of the above theorem (\cite[Thm.~1.2]{AFS17}) was used to construct a good moduli space of $\bQ$-Gorenstein smoothable, K-semistable Fano varieties.  Specifically, \cite[Thm.~1.2]{AFS17} states that if $\cX$ is an Artin stack of finite type with affine diagonal over $\bk$, then $\cX$ admits a good moduli space $\cX \to X$ if the following conditions hold: 
\begin{enumerate}
\item for every closed point $x \in \cX$, the stabilizer $G_x$ is reductive and there exists an \'etale morphism $f \co (\cW, w) \to (\cX,x)$ where $\cW \cong [\Spec(A)/G_x]$ such that 
\begin{enumerate}
	\item[(a)] $f$ induces an isomorphism of stabilizer groups at all closed points and
	\item[(b)] $f$ sends closed points to closed points, and
\end{enumerate}
\item for any $\bk$-point $y \in \cX$, the closure $\overline{ \{ y \}}$ admits a good moduli space.
\end{enumerate}
Vaguely speaking, condition (1a) ensures that the two projections $\cR:=\cW \times_{\cX} \cW \rightrightarrows \cW$ induce isomorphism of stabilizer groups while conditions (1b) and (2) ensure that the projections send closed points to closed points.  This is sufficient to imply that the two projections induce an \'etale equivalence relation $R \rightrightarrows W$ on good moduli spaces and that the algebraic space quotient $W/R$ is a good moduli space of $\cX$ Zariski-locally around $x$. 

We would like to explain the general idea of why the properties of S-completeness and $\Theta$-reductivity imply that the above conditions hold.  First, S-completeness implies that $G_x$ has a reductive stabilizer (Remark \ref{R:reductive}) and 
the existence of an \'etale morphism $f \co (\cW:=[\Spec(A)/G_x], w) \to (\cX,x)$ then follows from \cite[Thm.~1.2]{AHR15}.  %\jarodcut{(In our case, since we have a global quotient stack, the existence of an \'etale morphism $f \co (\cW:=[\Spec(A)/G_x], w) \to (\cX,x)$ follows from $G_x$ being reductive.)} \jarod{For normal quotients stacks (normal scheme mod $GL_n$), there is a direct (but still non-trivial) argument to show this.  In our case, $\kkX^{\rm Kss}_{n,V}$ is not necessarily normal, right?  So I suggest we remove this sentence.}

S-completeness  implies that after shrinking $\Spec(A)$, we may arrange that (1a) holds.  A complete argument is given in \cite[Prop.~4.4]{AHH18} but we explain here only how S-completeness implies that $f$ induces an isomorphism of stabilizer groups at any generization of $w$.  Let $\xi \co (\Spec(R), 0) \to (\cW, w)$ be a morphism from a complete DVR $R$ (with fraction field $K$). %Since $f$ is \'etale and $R$ is complete, this lifts uniquely to a morphism $\xi \co (\Spec(R), 0) \to (\cW, w)$.  
Then
$$\begin{aligned}
	\Aut_{\cW}(\xi_K) &\cong  \{\text{maps } g \co  \overline{\ST}_R \setminus 0 \to \cW \text{ and isomorphisms }  g|_{s \neq 0} \simeq \xi \simeq g|_{t \neq 0} \} \\
			& \cong \{\text{maps } g \co  \overline{\ST}_R \to \cW \text{ and isomorphisms }  g|_{s \neq 0} \simeq \xi \simeq g|_{t \neq 0} \}
\end{aligned}$$
where we have used S-completeness in the second line.  There is an analogous description of $\Aut_{\cX}( f(\xi_K) )$.  Since $f$ is \'etale and $R$ is complete, Tannaka duality implies that any map $(\overline{\ST}_R,0) \to (\cX,x)$ lifts uniquely to a map $(\overline{\ST}_R,0) \to (\cW,w)$.  It follows that $\Aut_{\cW}(\xi_K) \cong \Aut_{\cX}( f(\xi_K) )$.

Similarly, $\Theta$-reductivity implies that after shrinking $\Spec(A)$ further, we may arrange that (1b) holds.  A complete argument is given in \cite[Prop.~4.4]{AHH18} but we show here that if $\xi \in \cW$ is a generization of $w$ such that $\xi \in \cW_{K}$ is closed where $K = \overline{k(\xi)}$, then $\eta:=f(\xi) \in \cX_{K}$ is also closed.  Indeed, suppose $\eta \rightsquigarrow \eta_0$ is a specialization to a closed point in $\cX_{K}$; this can be realized by a map $\lambda \co \Theta_{K} \to \cX$. %(as $\xi$ has reductive stabilizer because $\cX$ is S-complete)
If $h \co \Spec(R) \to \cW$ is a map from a DVR with fraction field $K$ realizing the specialization $\xi \rightsquigarrow w$, then $\lambda$ and $f \circ h$ glue to form a map $\Theta_R \setminus 0 \to \cX$ which can be extended (using $\Theta$-reductivity) to a map $(\Theta_R,0) \to (\cX,x)$, and this in turn (using \'etaleness of $f$ and completeness of $R$) lifts to a unique map $(\Theta_R,0) \to (\cW,w)$.  But since $\xi \in \cW_{K}$ is closed, the image of $\Theta_K \to \cW$ consists of a single point, and thus the same is true for the image of $\lambda$.  It follows that $f(\xi)=\eta_0 \in \cX_K$ is closed.

Finally, both the S-completeness and $\Theta$-reductivity imply that (2) holds.  Let $y_0 \in \cY:=\overline{ \{y \}}$ be a closed point and  
 $f \co  (\cW:=[\Spec(A)/G_{y_0}], w_0) \to (\cY,y_0)$ be an \'etale morphism in which we can arrange that $w_0$ is the unique preimage of $y_0$.
 By Zariski's main theorem, we may factor $f$ as the composition of a dense open immersion $\cW \hookrightarrow \tilde{\cW}$ and a finite morphism $\tilde{\cW} \to \cY$.  Note that $w_0 \in \tilde{\cW}$ is necessarily closed 
 and that any other closed point  in $\tilde{\cW}$ is a specialization of a $k$-point in $\cW$.
 As $\tilde{\cW}$ is also $\Theta$-reductive, any $k$-point has a unique specialization to a closed point.  It follows that $w_0$ is the unique closed point in $\tilde{\cW}$ and thus the complement $\tilde{\cW} \setminus \cW$ is empty.  This in turn implies that $f \co \cW \to \cY$ is finite \'etale of degree $1$ and thus an isomorphism.

\medskip

In \cite{LWX14},  using analytic results, a stronger result than (2) was obtained, and as a result, the good moduli space is a scheme instead of merely an algebraic space. 
\end{rem}

\begin{lem}\label{l-substack}
 Let $f \co \cX \to \cY$ be a finite type monomorphism of Artin stacks locally of finite type over $\bk$ such that for every geometric point $x \co \Spec(l) \to \cX$, the image under $\cX_l \to \cY_l$ of the closure $\overline{ \{ x \} } \subset \cX_l$ is closed in $\cY_l$.  If $\cY$ is $\Theta$-reductive (resp., S-complete) with respect to essentially of finite type DVRs, then so is $\cX$.
\end{lem}

\begin{proof}
Zariski's main theorem implies that there is a factorization $f \co \cX \hookrightarrow \tilde{\cX} \to \cY$ where $\cX \hookrightarrow \tilde{\cX}$ is an open immersion and $\tilde{\cX} \to \cY$ is finite.  By \cite[Prop.~3.20(1)]{AHH18},  $\tilde{\cX}$ is also $\Theta$-reductive with respect to essentially of finite type DVRs, so may assume that $f$ is an open immersion.  Consider an essentially of finite type DVR $R$ with residue field $l = R/\pi$ and a morphism $h \co \Theta_R \setminus 0 \to \cX$.  Since $\cY$ is $\Theta$-reductive, $h$ extends to a diagram

$$\xymatrix{
\Spec(l) \ar[r]^{\pi=0} \ar@{^(->}[d]		& \Theta_R \setminus 0 \ar[r]^h \ar@{^(->}[d]				& \cX \ar@{^(->}[d]^f \\
\Theta_l \ar[r]^{\pi=0}		& \Theta_R \ar[r]^{\tilde{h}}						& \cY.
}$$
In particular, if $x$ denotes the composition $\Spec(l) \to \Theta_R \setminus 0 \to \cX$, we have a specialization $x \rightsquigarrow \tilde{h}(0)$ in $\cY_l$.  The hypotheses imply that $\tilde{h}(0) \in \cX_l$ so that $\tilde{h}$ factors though $\cX$.  The argument for S-completeness is analogous. 
\end{proof}

\subsection{Log Fano pairs and K-stability} 
In this section, we introduce some basic notions concerning log Fano pairs and K-stability.
For further background information, see \cite[Sect. 2]{BX18} and the references therein. 
\medskip

A \emph{pair} $(X,D)$ is composed of a normal variety $X$ and an effective $\Q$-divisor $D$
on $X$ such that $K_X+D$ is $\mathbb{Q}$-Cartier. 
See \cite[2.34]{KM98} for the definitions of \emph{klt}, \emph{plt}, and \emph{lc} pairs.
A pair $(X,D)$ is \emph{log Fano} if $X$ is projective, $(X,D)$ is klt, and $-K_X-D$ is ample. 
A variety $X$ is {\it $\mathbb{Q}$-Fano} if $(X,0)$ is log Fano.

\subsubsection{Families of log Fano pairs}

\begin{defn}
Let $T$ be a normal scheme.
A {\it $\mathbb{Q}$-Gorenstein family of log Fano pairs} $(X,D)\to T$ 
is composed of a flat projective morphism between normal schemes $X \to T$  and a $\Q$-divisor $D$ on $X$ satisfying:
\begin{enumerate}
\item $\Supp(D)$ does not contain any fiber,
\item $K_{X/T}+D$ is $\bQ$-Cartier, and
\item  $(X_{\overline{t}}, D_{\overline{t}} )$ is a log Fano pair for all $t\in T$. 
\end{enumerate}
In (3), $D_{\overline{t}}$ denotes the \emph{divisorial pullback} of $D$. 
More generally, if $S\to T$ is a morphism of normal schemes, we set $X_S := X\times_T S$ and  write $D_S$ for the $\Q$-divisor on $X_S$ associated to ${\rm Cycle}(D\times_T S)$.
\end{defn}

A \emph{special test configuration} of a log Fano pair $(X,D)$ is the data of a $\mathbb{G}_m$-equivariant $\Q$-Gorenstein family of log Fano pairs $(\cX,\cD) \to \mathbb{A}^1$ with an isomorphism $(\cX_1,\cD_1) \simeq (X,D)$ for $\{1\}\to \mathbb{A}^1$. 

\subsubsection{K-stability}
Let $(X,D)$ be an $n$-dimensional log Fano pair. A \emph{divisor over} $X$ is  a prime divisor $E$ on a normal variety $Y$ with a proper  birational morphism $\mu:Y \to X$.
Following \cite{Fuj18}, we  set 
$$
\beta_{X,D}(E)=(-K_X-D)^n A_{X,D}(E) - \int_0^\infty  \vol(\mu^*(-K_X-D)- t E)\, dt,$$
where $A_{X,D}(E) : = 1+ {\rm coeff}_E(K_{Y} - \mu^*(K_X+D))$ is the log discrepancy.

\begin{defn}\label{d-kstable}
A log Fano pair $(X,D)$ is
\begin{enumerate}
\item  \emph{K-semistable} if $\beta_{X,D}(E)\ge 0$ for all divisors  $E$ over $X$;
\item \emph{K-stable} if $\beta_{X,D}(E)> 0$ for all divisors  $E$ over $X$;
\item \emph{K-polystable} if it is K-semistable and  for any special 
test configuration of ${(\cX,\cD) \to \mathbb{A}^1}$ of $(X,D)$ with $(\cX_0, \cD_0)$ 
K-semistable there is an isomorphism of $\Q$-Gorenstein families of log Fano pairs   $ (\cX, \cD) \simeq (X_{\mathbb{A}^1}, D_{\mathbb{A}^1}):=(X,D)\times\mathbb{A}^1$.
 \end{enumerate}
 \end{defn}
% \HB{The previous version I wrote was a bit vague. I want to make it clear that the above isomorphism is not an equivariant isomorphism between $(\cX,\cD)$ and the trivial test configuration}

%\dan{I think the definition should be the one which is a-priori easier to check, i.e. formulated only in terms of special test configurations, rather than arbitrary divisors over $X$. The finite generation in part (3) is especially intimidating as part of a definition.}
The equivalence of the above definition with the original definitions in \cite{Tia97, Don02} was proven in \cite{Fuj16,Li17, LWX18, BX18}.

\medskip
Though the above notions of stability make sense for log Fano pairs over characteristic zero 
fields that are not algebraically closed, we will not use them due to the following issue:
Let $(X_K,D_K)$ be a log Fano pair over a characteristic zero field $K$
and $K'/K$ a field extension. 
While it is expected that  $(X_K,D_K)$ is K-semistable if and only if $(X_{K'},D_{K'})$ is K-semistable,  the result is only known when both $K$ and $K'$ are algebraically closed 
(for example, see \cite[Cor. 15]{BL18}). \footnote{Since the first version of the current paper was written, this expectation was proved in \cite{Zhu20}.}

The following result proved in \cite{LWX18} will be needed in various places.

\begin{lem}[{\cite{LWX18}*{Lem. 3.1}}]\label{l-sequivalent}
Let $(\cX, \cD)$ be a special test configuration of a K-semistable log Fano pair
$(X,D)$ with the central fiber denoted by $(X_0, D_0)$. If $\Fut(\cX,\cD) = 0$, then $(X_0, D_0)$ is K-semistable.
\end{lem}

\subsection{Flat families of polarized schemes over a surface}
We will be considering $S$-completeness and $\Theta$-reductivity of stacks parameterizing polarized varieties. Both conditions are formulated in terms of the existence of extensions of equivariant flat families of polarized varieties over punctured regular surfaces.

We thus consider a regular noetherian 2-dimensional scheme $S$, and a closed point $0 \in S$. Let $j : S \setminus 0 \to S$ be the open immersion. The key fact that we will use is that for any finite rank locally free sheaf $E$ on $S\setminus 0$, $j_\ast(E)$ is locally free as well. $j_\ast(E)$ is coherent because $S$ is normal and $0$ has codimension $2$, and the reflexive sheaf $j_\ast(E)$ is locally free because any reflexive sheaf on a regular $2$-dimensional scheme is locally free \cite{hartshorne}*{Cor.~1.4}. More precisely, $j_\ast$ induces an equivalence between the categories of locally free (and more generally, flat quasi-coherent) sheaves on $S \setminus 0$ and on $S$ locally free (respectively, flat quasi-coherent) sheaves on $S$, with inverse given by restriction.

\begin{lem} \label{L:filling_conditions}
Let $q : \cX \to S \setminus 0$ be a flat projective morphism of schemes, and let $\cL$ be a relatively ample line bundle on $\cX$. Then the following are equivalent:
\begin{enumerate}
\item there exists an extension of $q$ to a flat projective family $\widetilde{\cX} \to S$ with an ample $\bQ$-line bundle $\tilde{\cL}$ extending $\cL$;
%\footnote{ \color{blue} When we say line bundle here, are we thinking about ``$\Q$-line bundles''? (i.e. $\cO_{\tilde{\cX}}(1)$ may not be a line bundle if the graded ring is not finitely generated in degree 1}
\item the algebra $\bigoplus_{m \geq 0} j_\ast (q_\ast( \cO_\cX(m \cL)))$ is finitely generated as an $\cO_S$-algebra; and
\item the restriction $\bigoplus_{m \geq 0} j_\ast (q_\ast( \cO_\cX(m \cL)))|_{0}$ is finitely generated as a $\kappa(0)$-algebra.
\end{enumerate}
If these conditions hold, then $$\widetilde{\cX} = {\rm Proj}_S\left(\bigoplus_{m \geq 0} j_\ast \big(q_\ast \cO_\cX(m \cL)\big)\right)$$ is the unique extension, with the polarization $\cO_{\widetilde{\cX}}(1)$. If $\cX$ is equivariant for an action of $\bG_m$ on $S$, then so is $\tilde{\cX}$.
\end{lem}
%\footnote{\HB{Do we think it is fine to abuse notation here? (In birational geometry it is common to not distinguish between line bundles and Cartier divisors... though, it is a bit unusual to write $\cO_{\cX}(\cL)$ when $\cL$ is a line bundle.)}}
\begin{proof}
$(1) \Leftrightarrow (2)$: Note that $q_\ast(\cO_{\cX}(m\cL))$ is locally free on $S$ for $m \gg 0$ because $q$ is flat. It follows that $\tilde{\cX} = {\rm Proj}_S (\bigoplus_m j_\ast(q_\ast\cO_{\cX}(m\cL)) )$ is a flat extension of $\cX$ if this algebra is finitely generated, and conversely for any flat extension $\Gamma(\widetilde{\cX},\cO_{\widetilde{\cX}}(m\widetilde{\cL})) = j_\ast(q_\ast(\cO_{\cX}(m\cL)))$ for $m\gg 0$.

$(3) \Leftrightarrow (2)$: Note that $(2) \Rightarrow (3)$ automatically, and finite generation is local over $S$ by definition, so we may assume $S$ is affine. Then we may lift a finite homogeneous set of generators of $\bigoplus_{m \geq 0} j_\ast(q_\ast(\cO_{\cX}(m\cL))) \otimes_{\cO_S} \kappa(0)$ to $\bigoplus_{m \geq 0} j_\ast(q_\ast(\cO_{\cX}(m\cL)))$, and by assumption we may find homogeneous elements in the latter which generate the algebra $\bigoplus_{m \geq 0} q_\ast(\cO_{\cX}(m\cL))$ after restriction to $S \setminus 0$. Together these define a map of graded $\cO_S$-algebras $\phi : \cO_S[x_1,\ldots,x_N] \to \bigoplus_{m \geq 0} j_\ast(q_\ast(\cO_{\cX}(m\cL)))$, where the degree of the generators $x_i$ vary but are all $\geq 0$. $\phi$ is surjective after restriction to $\kappa(0)$ and $S\setminus 0$, so because the graded pieces of both algebras are finite $\cO_S$-modules, Nakayama's lemma implies that $\phi$ is surjective.

Note that if $\widetilde{\cX}$ is equivariant for a $\bG_m$-action on $S$, then $\bigoplus_m j_\ast(q_\ast(\cO_{\cX}(m \cL)))$ has an additional grading coming from the $\bG_m$-action, and this grading induces a $\bG_m$-action on $\widetilde{\cX}$ extending the one on $\cX$.
\end{proof}

\section{S-completeness}\label{s-Scomplete}

In this section, we will prove that the moduli of K-semistable log Fano pairs is $S$-complete (Theorem \ref{t-scomplete}).
 We first study S-completeness for quasi-coherent sheaves in Section \ref{ss-scoherent} and then S-completeness of polarized varieties in Section \ref{ss-spolar}. Applying this to the direct sum of the pushforwards of  the $m$-th tensor product of the polarization for a family of polarized varieties, this naturally leads to a finite generation condition on the graded algebra (see Condition \ref{a-finitegenerate}). In Section \ref{ss-sfano}, we confirm this condition for K-semistable log Fano pairs.

\subsection{$S$-completeness for coherent sheaves}\label{ss-scoherent}  In this subsection, we establish S-completeness for the stack parameterizing coherent sheaves on $\Spec(k)$ or, in other words, that every flat and coherent sheaf on $\overline{\ST}_R \setminus 0$ extends uniquely to a flat and coherent sheaf on $\overline{\ST}_R$.

We begin by discussing the correspondence between flat coherent sheaves on $\Theta_k$ and filtrations.  A quasi-coherent sheaf $F$ on $\Theta_k = [\Spec(k[x])/\bG_m]$ corresponds to a $\bG_m$-equivariant quasi-coherent sheaf on $\Spec(k[x])$ or, in other words, a $\bZ$-graded $k[x]$-module $\bigoplus_{p \in \bZ} F_p$; this in turn corresponds to diagram of $k$-vector spaces:  $\cdots \to F_{p+1} \xrightarrow{x} F_p \xrightarrow{x} F_{p-1} \to \cdots$.  The restriction of $F$ along $\Spec(k) \xrightarrow{1} \Theta_k$ is $\colim(\cdots \to F_{p+1} \xrightarrow{x} F_p \to \cdots)$ and along $B_{k} \bG_m \xrightarrow{0} \Theta_k$ is the associated graded quasi-coherent sheaf $\bigoplus_p F_p / xF_{p+1}$.  Moreover, $F$ is flat and coherent over $\Theta_k$ if and only if each $F_p$ is a finite dimensional $k$-vector space,  the maps $x$ are injective, $F_p = 0$ for $p \gg 0$ and $x \co F_p \to F_{p-1}$ is an isomorphism for $p \ll 0$.
 
 Similarly, if $R$ is a DVR with fraction field $K$, residue field $\kappa$ and uniformizing parameter $\pi$, then a quasi-coherent sheaf $F$ on $\overline{\ST}_R =  [\Spec \big(R[s,t] / (st-\pi)\big)  / \bG_m]$ corresponds to a $\bG_m$-equivariant quasi-coherent sheaf on $\Spec(k[s,t]/(st-\pi))$ or, in other words, a
  $\bZ$-graded $R[s,t]/(st-\pi)$-module $\bigoplus_{p \in \bZ} F_p$; this in turn corresponds to a diagram of maps of $R$-modules 
$$
\xymatrix{\cdots \ar@/^/[r]^t & F_{p+1} \ar@/^/[r]^t \ar@/^/[l]^s & F_{p} \ar@/^/[r]^t \ar@/^/[l]^s & F_{p-1}  \ar@/^/[r]^t \ar@/^/[l]^s & \ar@/^/[l]^s \cdots },
$$
such that $st=ts=\pi$.  The reader may wish to refer to the schematic picture \eqref{E:ST-schematic}  of  $\overline{\ST}_R$.  The restriction of $F$ along
	\begin{itemize}
		\item  $\Spec(R) \xhookrightarrow{t \neq 0} \overline{\ST}_R$ is
		$\colim (\cdots \xrightarrow{t} F_{p} \xrightarrow{t} F_{p-1} \xrightarrow{t} \cdots)$,
		\item  $\Spec(R) \xhookrightarrow{s \neq 0} \overline{\ST}_R$ is
		$\colim (\cdots \xleftarrow{s} F_{p} \xleftarrow{s} F_{p-1} \xleftarrow{s} \cdots)$,
		\item  $\Theta_{\kappa} \xhookrightarrow{t = 0} \overline{\ST}_R$ is the object corresponding to the sequence 
		$$(\cdots \xleftarrow{s} F_p/tF_{p+1} \xleftarrow{s} F_{p-1}/tF_p \xleftarrow{s} \cdots),$$ 
		\item  $\Theta_{\kappa} \xhookrightarrow{s = 0} \overline{\ST}_R$ is $(\cdots \xrightarrow{t} F_{p+1}/sF_{p} \xrightarrow{t} F_{p}/sF_{p-1} \xrightarrow{t} \cdots)$, and
		\item  $B_{\kappa} \bG_{m} \xhookrightarrow{s=t = 0} \overline{\ST}_R$ is the $\bZ$-graded $\kappa$-module $\bigoplus_{p \in \bZ} F_p/(t F_{p+1} + s F_{p-1})$.
	\end{itemize}
%Moreover, $F$ is coherent if each $F_p$ is coherent, $s \co F_{p} \to F_{p-1}$ is an isomorphism for $p \ll 0$ and $t \co F_{p-1} \to F_{p}$ is an isomorphism for $p \gg 0$.  
The sheaf $F$ is flat and coherent over $\overline{\ST}_R$ if and only if each $F_p$ is flat and coherent over $R$, the maps $s$ and $t$ are injective, the induced maps $s \co F_{p-1}/tF_p \to F_p/tF_{p+1}$ are injective (or equivalently the maps $t \co F_{p+1}/sF_p \to F_p/sF_{p-1}$ are injective), $t \co F_{p} \to F_{p-1}$ is an isomorphism for $p \ll 0$ and $s \co F_{p-1} \to F_{p}$ is an isomorphism for $p \gg 0$. 
	
Let $j \co \overline{\ST}_R \setminus 0 \hookrightarrow \overline{\ST}_R$ be the open immersion.  We will show how to compute the pushforward of coherent sheaves under this open immersion.
Let $j_t, j_s \co \Spec(R) \to \overline{\ST}_R$ and $j_{st} \co \Spec(K) \to \overline{\ST}_R$ be the open immersions corresponding to $t \neq 0$, $s \neq 0$ and $st \neq 0$.  
Let $\cE$ be a flat coherent sheaf on $\overline{\ST}_R \setminus 0$; this corresponds to 	a pair of $R$-modules $E$ and $E'$ together with an isomorphism $\alpha \co E_{K} \to E'_{K}$.	Under $\alpha$, we may view both $E$ and $E'$ as submodules of $E_K$.  Then 
%\jarodchange{$j_\ast \cE = \ker\big( (j_s)_\ast E \oplus (j_t)_\ast E' \to (j_{st})_\ast(E_K) \big).$}
{$j_\ast \cE \cong  (j_t)_\ast E \cap (j_s)_\ast E' \subset (j_{st})_\ast E_K$.}
As morphisms of graded $R[s,t]/(st - \pi)$-modules, $j_t$ and $j_s$ correspond to the inclusions $R[s,t]/(st-\pi) \subset R[t]_t$ and $R[s,t]/(st-\pi) \subset R[s]_s$, and $j_{st}$ corresponds to $R[s,t]/(st-\pi) \subset K[t]_t$.  Recalling that $t$ has weight $-1$, we 
 compute that
	$$\begin{aligned}
	(j_{st})_\ast E_K &\cong E_K  \otimes_R R[t]_t  \cong \bigoplus_{p\in \bZ} E_K t^{-p}, \\
	(j_t)_\ast E & \cong E \otimes_R R[t]_t \cong \bigoplus_{p \in \bZ} E t^{-p}  \subset (j_{st})_\ast E_K,  \\
	(j_s)_\ast E' & \cong E' \otimes_R R[s]_s \cong \bigoplus_{p \in \bZ} (\pi^{p} \cdot E') t^{-p} \subset (j_{st})_\ast E_K
	\end{aligned}$$
	where in the last line we have used the identification $s = t^{-1} \pi$.  Finally, we compute that
\begin{equation} \label{E:push-forward}
j_\ast \cE \cong \bigoplus_{p \in \bZ} \big( E \cap (\pi^{p} \cdot E') \big) t^{-p} \subset \bigoplus_{p\in \bZ} E_K t^{-p}.
\end{equation}
If we define the filtration $\cG^pE = E \cap (\pi^{p} \cdot E')$, then $j_\ast \cE $ is the $\oh_{\overline{\ST}_R}$-module given by the diagram
  \begin{equation*}
\xymatrix{\cdots \ar@/^/[r]^t & \cG^{p+1}E \ar@/^/[r]^t \ar@/^/[l]^s & \cG^pE \ar@/^/[r]^t \ar@/^/[l]^s & \cG^{p-1}E  \ar@/^/[r]^t \ar@/^/[l]^s & \ar@/^/[l]^s \cdots },
\end{equation*}
 of $R$-modules where $t \co \cG^{p+1}E \to \cG^pE$ is inclusion and $s \co \cG^pE \to \cG^{p+1}E$ is multiplication by $\pi$.  Note that $j_* \cE$ is necessarily a flat and coherent $\oh_{\overline{\ST}_R}$-module, because non-equivariantly it is the pushforward of a vector bundle from the complement of a closed point in the regular surface $\Spec(R[s,t]/(st-\pi))$.

\subsection{S-completeness for polarized varieties}\label{ss-spolar}

Suppose $(X,L)$ and $(X',L')$ are flat families of polarized varieties over $\Spec(R)$ and $\alpha \co (X_K, L_K) \to (X'_K, L'_K)$ is an isomorphism.  Then $(X,L)$ and $(X',L')$ can be glued along the isomorphism $\alpha$ to a polarized family $(\cX, \cL) \to \overline{\ST}_R \setminus 0$.  This yields a diagram
$$\xymatrix{
\cX \ar[d]^{q}	\\
\overline{\ST}_R \setminus 0 \ar[r]^j	& \overline{\ST}_R \, .
}$$

Now we state our key condition:
\begin{condition}[Finite Generation Condition]\label{a-finitegenerate}
The  $\oh_{\overline{\ST}_R}$-algebra 
$\bigoplus_{m \ge 0} j_* q_* \oh_{\cX}(m\cL)$  is finitely generated.
\end{condition}
By Lemma \ref{L:filling_conditions}, this condition is equivalent to the existence of a flat extension of $\cX$ to a polarized family $(\widetilde{\cX}, \widetilde{\cL}) \to \overline{\ST}_R$, where
%Under the Finite Generation Assumption \ref{a-finitegenerate}, we can take
$$\tilde{\cX} := \uProj_{\, \overline{\ST}_R} \bigoplus_{m \ge 0} j_* q_* \oh_{\cX}(m\cL)$$
and $\tilde{\cL} = \oh_{\tilde{\cX}}(1)$.  To provide a more explicit description for this algebra, Equation \eqref{E:push-forward} implies that for each $m \ge 0$, 
$$\begin{aligned}
j_* q_*  \oh_{\cX}(m\cL) & \cong \bigoplus_{p \in \bZ} \big(H^0(X, \oh_X(mL)) \cap \pi^{p} H^0(X', \oh_{X'}(mL')) \big)t^{-p}  \\
	& \subset \bigoplus_{p \in \bZ} H^0(X_K, \oh_{X_K}(m L_K)) t^{-p}.
\end{aligned}$$

Define a filtration of $V_m := H^0(X, \oh_X(mL))$ by 
$$\cG^pV_m := H^0(X, \oh_X(mL)) \cap \pi^{p} H^0(X', \oh_{X'}(mL')),$$ which consists of sections in $V_m$ with at worst a pole of order $p$ along $X'_0$.  We have a diagram of $R$-modules
  \begin{equation*}
\xymatrix{\cdots \ar@/^/[r]^t & \cG^{p+1}V_{m} \ar@/^/[r]^t \ar@/^/[l]^s & \cG^pV_{m} \ar@/^/[r]^t \ar@/^/[l]^s & \cG^{p-1}V_{m}  \ar@/^/[r]^t \ar@/^/[l]^s & \ar@/^/[l]^s \cdots },
\end{equation*}
where $t \co \cG^{p+1}V_{m} \to \cG^pV_{m}$ is inclusion and $s \co \cG^pV_{m} \to \cG^{p+1}V_{m}$ is multiplication by $\pi$.  This gives the direct sum $\bigoplus_{p,m} \cG^pV_{m}$ the structure of a bigraded $R[s,t]/(st-\pi)$-algebra.  Assume the Finite Generation Condition \ref{a-finitegenerate} holds, then the grading in $m$ defines a projective scheme 
$$\cP = \uProj_{\Spec(R[s,t]/(st-\pi))} \bigoplus_{p,m} \cG^pV_{m}$$ and the grading in $p$ gives an action of $\bG_m$ on $\cP$ and a linearization of $\oh_{\cP}(1)$.  Observe that $(\tilde{\cX}, \tilde{\cL})= ([\cP / \bG_m], \oh_{\cP}(1))$.

\begin{example}
Let $(X, L)$ be a polarized $\kappa$-variety, and let $R = \kappa [\![t]\!]$ and $K=\kappa(\!(t)\!)$.  Let $(X_K, L_K) \to (X_K,L_K)$ be an automorphism induced from a one-parameter subgroup $\alpha \co \bG_m \to \Aut(X,L)$.  The above construction produces a flat family $(\tilde{\cX}, \tilde{\cL})$ over $\overline{\ST}_R$ which corresponds to the trivial flat family 
$$\big(X \times \Spec(R[s,t]/(st-\pi)), p_1^*L\big)$$ 
over $\Spec( R[s,t]/(st-\pi) )$ with the $\bG_m$-action given by $\alpha$ on the first factor.  Observe that if $\Aut(X,L)$ is reductive, then any $\alpha \in \Aut(X,L)(K)$ is in the same double coset as a one-parameter subgroup by the Iwahori decomposition, and it follows that any family over $\overline{\ST}_R \setminus 0$ obtained by gluing two trivial families over $\Spec(R)$ along an isomorphism $\alpha \in \Aut(X,L)(K)$ extends to a family over $\overline{\ST}_R$. On the other hand if $\Aut(X,L)$ is not reductive, such an extension need not exist.
%Therefore the construction of  $(\tilde{\cX}, \tilde{\cL})$ for such an $\alpha$ leads also to the trivial family over $\Spec R[s,t]/(st-\pi)$ with a $\bG_m$-action. \dan{I'm not sure what this last sentence is saying}  However, if $\Aut(X,L)$ is not reductive, this construction associated to an $\alpha \in \Aut(X,L)(K)$ may produce a polarized family $\tilde{\cX}$ where $\tilde{\cX}_0$ is a degeneration of $X$ \dan{It may also fail to have the finite generation hypothesis, right?}. 
\end{example}

%\begin{example}
%<<TODO: include a non-fillable example>>
%\end{example}

\subsection{S-completeness for K-semistable log Fano pairs}\label{ss-sfano}
In this section, we will prove that Condition \ref{a-finitegenerate} holds 
for K-semistable log Fano pairs with anticanonical polarization (Theorem \ref{t-scomplete}).
This is obtained by showing that the filtration considered in \cite{BX18} is equivalent  
 to the filtration in Section \ref{ss-spolar} up to a grading shift.
Hence, we can invoke finite generation results proved in \cite{BX18}
 to verify  that Condition \ref{a-finitegenerate} is satisfied
 and then use a result in \cite{LWX18} (see Lemma \ref{l-sequivalent}) to show that the  corresponding special fiber 
of the flat extension over  $ \overline{\ST}_R$ is K-semistable. 

\medskip 
Let $R$ be a DVR essentially of finite type over $k$ with uniformizer $\pi$, 
fraction field $K$, 
and residue field $\kappa$. 
Let
\[
(X,D) \to \Spec(R) 
\quad \text{ and } \quad
(X',D') \to \Spec(R)
\]
be $\Q$-Gorenstein families of log Fano pairs
and assume there is a birational map $\alpha \colon X \dashrightarrow X'$
that induces an isomorphism $(X_K, D_K) \to (X'_K,D'_K)$.  
%but does not extend to an isomorphism $X\to X'$.  
Following Section \ref{ss-spolar}, 
the above data gives a $\mathbb{G}_m$-equivariant
 $\Q$-Gorenstein family of log Fano pairs 
\begin{equation}\label{e-punctured}
(\cX , \cD )
 \to 
 \Spec \left(R[s,t]/(st-\pi) \right) \setminus 0,  
\end{equation}
where $0 \in  \Spec \left(R[s,t]/(st-\pi) \right)$ is the closed point defined by the vanishing of $(s,t)$. 

\begin{thm}\label{t-scomplete}
If $(X_{\overline{\kappa}}, D_{\overline{\kappa}})$ and $(X'_{ \overline{\kappa}}, D'_{\overline{\kappa}})$ are K-semistable, 
then the map in \eqref{e-punctured} extends uniquely to a $\mathbb{G}_m$-equivariant $\Q$-Gorenstein  family of log Fano pairs 
\[
(\tilde{\cX}, \tilde{\cD}) \to  \Spec\left(R[s,t]/(st-\pi)\right)
.\]
Furthermore, the geometric fiber over $0$ is K-semistable.
\end{thm}

 \begin{rem}

 \begin{enumerate}
 \item  The above theorem immediately implies that $\kkX^{\rm Kss}_{n,V}$ is S-complete with respect to essentially of finite type DVRs.
 
 \item
Theorem \ref{t-scomplete} is an extension of \cite[Thm 1.1.1]{BX18}, 
which states that if $(X_{\overline{\kappa}}, D_{\overline{\kappa}})$ 
and $(X'_{ \overline{\kappa}}, D'_{\overline{\kappa}})$ are K-semistable,
then they degenerate to a common K-semistable log Fano pair via special
 test configurations. Indeed, the restriction of 
 $({\cX},  {\cD}) \to \Spec\left(R[s,t]/(st-\pi)\right) $ 
 to $s=0$ and $t=0$ are naturally test configurations of $(X_\kappa,D_\kappa)$
 and $(X'_\kappa,D'_\kappa)$ with special fiber $(\tilde{\cX}_{0}, \tilde{\cD}_{0})$.
 
\item
The results in \cite{BX18} are phrased in the setting of families
 over a smooth pointed curve, not the spectrum of a DVR. Fortunately, 
 the proofs in \cite[Sect. 5]{BX18} extend with little change to the more general setting of families 
 over DVRs which are essentially of finite type over $k$.  

However, the argument does not automatically 
generalize to families over the spectrum of a general DVR over $k$,
since a key ingredient in the proof relies on the MMP, specifically \cite{BCHM10}.
While the latter results hold for varieties 
(and, hence, have natural extensions to essentially of finite type $k$-schemes),
they are not known to hold more generally.
\end{enumerate}
 \end{rem}

\subsubsection{Filtration from \cite{BX18}}\label{s-vanishingfilt}
%\footnote{\HB{This section was previously titled ``Vanishing order filtration,'' which may be confusing (since the filtration  from Section 3.2 is also related to vanishing orders)}}
Consider a diagram over $\Spec(R)$
\begin{center}
\begin{tikzcd}
  & Y \arrow[dr, "\rho'"] \arrow[dl,swap,"\rho"]& \\
  X  \arrow[rr, dashrightarrow, "\phi'"] & & X'
   \end{tikzcd} ,
   \end{center}
where $\rho$ and $\rho'$ are proper birational morphisms and $Y$ is normal.
Write $\widetilde{X}_0$ and $\widetilde{X}'_0$ for the birational transforms of $X_0$ and $X'_0$ on $Y$.
%Choose a canonical divisor $K_Y$ on $Y$, and set $K_X:= \rho_* (K_Y)$ and $K_{X'}:= \rho'_* (K_Y)$

Fix a positive integer $r$ such that $L:=-r(K_X+D) $ and $L' : = -r(K_{X'}+D')$ are Cartier divisors.
 Let
\[
V:=  \bigoplus_{m \in \N} V_m := \bigoplus_{m \in \N} H^0(X, \cO_{X}(mL )) 
 \quad 
 \text{ and } 
 \quad
 V':=  \bigoplus_{m \in \N} V'_m := \bigoplus_{m \in \N} H^0(X', \cO_{X'}(mL' )) 
\]
denote the section rings of $X$ and $X'$ with respect to $L$ and $L'$.
We write $V_{\kappa}= \bigoplus_m V_{\kappa,m}$
and
$  V_{K}= \bigoplus_m V_{K,m}$ 
for the restrictions of  $V$ to $\Spec(\kappa)$ and $\Spec(K)$, respectively.
Note that each $V_m$ is a flat $R$-module and satisfies cohomology and base change, 
since  $H^i(X, \cO_{X}(mL) )=0$ for $i>0$ and $m\geq 0$ by \cite[Thm. 10.37]{Kol13}. 
%of course, this also holds by Serre Vanishing if r>>0. 
Therefore, $V_{\kappa}$ and $V_K$ are isomorphic to the section rings of $L_{\kappa}$ and $L_{K}$. 

Following \cite[Sect. 5.1]{BX18}, for each $m \in \N$ and $p \in \Z$, we set 
\begin{equation}\label{e-filtBX}
\cF^{p} V_m : = \{ g \in V_{m} \, \vert\, \ord_{\widetilde{X}'_0} (g) \geq p   \},
\end{equation}
where
 $ \ord_{\tilde{X}'_0} (g)$ equals the coefficient of $\widetilde{X}'_0$ in  ${\rm div} (\rho^*(g))$.
 Observe that 
\begin{equation}\label{e-piFp}
\pi \cF^{p-1} V_m =  \cF^p V_m \cap \pi V_m
\end{equation}
and setting
$$\cF^p V_{\kappa, m} := {\rm im}( \cF^p V_m \otimes_R \kappa \to V_{m,\kappa} ) \subseteq V_{\kappa,m}
,$$
gives a filtration of the section ring $V_\kappa$.  
We state two results from \cite[Section 5.2]{BX18} concerning this filtration.

\begin{prop}\label{p-finitegen}
If  $(X_{\overline{\kappa}}, D_{\overline{\kappa}})$ 
and $(X'_{\overline{\kappa}}, D'_{\overline{\kappa}})$ are K-semistable, then:
\begin{enumerate}
\item 
The $\kappa[t]$-algebra 
 $\displaystyle \bigoplus_{m\in \N} \bigoplus_{p\in \Z}  \left(\cF^p V_{\kappa,m} \right) t^{-p}
 $ and $\kappa$-algebra  $\displaystyle\bigoplus_{m\in \N} \bigoplus_{p\in \Z}  \gr_{\cF}^p V_{\kappa,m}$
  are finitely generated;
  \item
   The test configuration  $(\cX_\kappa,\cD_\kappa) \to \mathbb{A}^1_{{\kappa}}$ 
   of $(X_\kappa, D_\kappa)$ induced by the $\kappa[t]$-algebra in (1)
   is special and  
  % \jarodchange{$(\cX_{\overline{\kappa}}, \cD_{\overline{\kappa}})$}
   {the geometric fiber over $0$}  is K-semistable.
  \end{enumerate}
  \end{prop}
  
  \begin{proof}
  The argument in \cite[Sect. 5.2]{BX18} implies (1) 
  and that the induced test configuration $(\cX_\kappa,\cD_\kappa) \to \mathbb{A}^1_\kappa$ of $(X_\kappa,D_\kappa)$ is a special test configuration with  Futaki invariant zero.
 Since  $\Fut(\cX_{\overline{\kappa}}, \cD_{\overline{\kappa}})= \Fut( \cX_{\kappa}, \cD_\kappa)$ and the latter is zero,  $( \cX_{\overline{\kappa}}, \cD_{\overline{\kappa}})_0$ must be K-semistable by Lemma \ref{l-sequivalent}.
 \end{proof}

In the proof of Theorem \ref{t-scomplete}, we will need to show that the boundary divisor $\cD$ in
\eqref{e-punctured} extends to a well defined family of cycles over $\Spec\left( R[s,t]/(st-\pi) \right)$. 
For this, let $B$ be a prime divisor in $\Supp(D)$ and write $I_B \subseteq \bigoplus_m V_m $ for the homogenous ideal defining $B$. Consider the homogenous ideal
\begin{equation}\label{e-Idef}
I : =
 \bigoplus_{m \in \N} \bigoplus_{p \in \Z} {\rm im} \left( I_B  \cap \cF^p V_m \to \gr_{\cF}^p V_{\kappa,m} \right) 
 \subseteq \underset{m \in \N}{\bigoplus} \underset{p \in \Z}{\bigoplus} \gr_{\cF}^p V_{\kappa,m}
.\end{equation}

\begin{prop}\label{p-cod1}
If $(X_{\overline{\kappa}}, D_{\overline{\kappa}})$ and $(X'_{\overline{\kappa}}, D'_{\overline{\kappa}})$
are K-semistable, 
 then the subscheme defined by $I$ is of codimension at least one. 
\end{prop}

\begin{proof}
Let $(\cX_\kappa, \cD_\kappa) \to \mathbb{A}_\kappa^1$ denote the test configuration described in Proposition \ref{p-finitegen}
and write $\cB_\kappa$ for the closure of $B_\kappa \times (\mathbb{A}^1 \setminus 0)$ in $\cX_\kappa$ under the imbedding   ${X_\kappa \times (\mathbb{A}^1 \setminus 0) \hookrightarrow \cX_\kappa}$. 
Clearly, the scheme theoretic fiber of $\cB_\kappa$ over $0$ is of codimension one in  $(\cX_\kappa)_0 \simeq \Proj( \bigoplus_m \bigoplus_p \gr_{\cF}^{p} V_{\kappa,m})$.
Since $V(I)$  and the scheme theoretic fiber of $\cB_\kappa$ over $0$ agree away from a  codimension 2 subset  by \cite[Prop. 5.13.1]{BX18}, $V(I)$ is also of codimension at least one.
\end{proof}
   
In light of the discussion in Section \ref{ss-spolar}%Jarod's section on S-completeness for coherent sheaves
, 
observe that
\begin{equation}\label{e-ReesF}
 \underset{m\in \N } {\bigoplus} \underset{ p \in \Z }{\bigoplus} \left( \cF^p V_m \right) t^{-p}
 \end{equation}
has the structure of a $\Z$-graded $R[s,t]/(st-\pi)$-module, where the map $  \left(\cF^p V_m \right)  t^{-p} \overset{s}{\to}   \left(\cF^{p+1} V_m\right) t^{-p-1}$ is defined by $ g t^{-p}\mapsto  \pi gt^{-p-1}$. 
Additionally,  
\begin{equation}\label{e-restricts=0}
\Big( \bigoplus_{m\in \N} \bigoplus_{p\in \Z}  \left(\cF^p V_m \right) t^{-p}  \Big) \bigotimes_{R[s,t]/(st-\pi)} \kappa[t]
\simeq 
 \bigoplus_{m\in \N} \bigoplus_{p\in \Z} \left(\cF^p V_{\kappa,m} \right) t^{-p} 
,\end{equation}
since \eqref{e-piFp} implies 
$\frac{\cF^p V_{m }}{
\pi \cF^{p-1} V_m } = \frac{\cF^p V_m}{ \pi V_m \cap \cF^p V_m } 
\simeq 
 {\rm im} \left( \cF^p V_m \to \frac{V_{m}}{\pi V_m} \right).
 $
Therefore, 
\begin{equation}\label{e-fiberover0}
\Big( \bigoplus_{m\in \N} \bigoplus_{p\in \Z}   (\cF^p V_m) t^{-p} \Big) \bigotimes_{R[s,t]/(st-\pi)} \kappa \, 
\simeq  \, \bigoplus_{m \in \N} \bigoplus_{p \in \Z} \gr_{\cF}^p V_{\kappa,m},
\end{equation}
where $R[s,t]/(st-\pi) \to \kappa$ is the morphism  that sends $s$ and $t$ to $0$.
\medskip

The following proposition states that the filtration $\cF^\bullet$ from \cite{BX18} coincides with the filtration from Section \ref{ss-spolar} up to a shift. 
See \cite[Section 2.5]{BHJ17}, \cite[Claim 5.4]{Fuj16} or \cite[(64)]{Li17} for related arguments applied to test configurations.

\begin{prop}\label{p-comparefilt}
For each $p\in \Z$ and $m \in \N$, 
\[
\cF^{p-mra} V_m= V_m \cap \pi^{p}V'_m,
\]
where $a :={\rm coeff}_{\widetilde{X}'_0} \left(K_Y - \rho^*(K_X+D) \right)$.
\end{prop}

The intersection in the above proposition is taken after using the isomorphism $\alpha^*: K(X') \to K(X)$
to view  $V'_m:= H^0(X', \cO_{X'}(mL'))$ as a  $R$-submodule of $K(X)$.

\begin{proof}
First, observe that there are natural isomorphisms 
 \begin{multline*}
 \pi^{p}H^0\left(X', \cO_{X'}(mL') \right) \simeq 
  H^0\left(X', \cO_{X'}(mL' -p X'_0 ) \right) \\
  \simeq 
 H^0\left(Y, \cO_Y \left(  {\rho'}^*(m L' -p X'_0) \right)\right) \\
  = H^0 \left( Y, \cO_Y \left( m \rho^*L + m  ( {\rho'}^* L' -\rho^*L) - p {\rho'}^* X'_0 \right) \right).
  \end{multline*}
  Next, fix $g \in H^0(X, \cO_X(mL))$ and set  $G= {\rm div}( g) $.
By the above isomorphisms, $g  \in  \pi^{p} H^0(X', \cO_{X'}(mL'))$ if and only if 
\[
G' := 
\rho^*G + \left( m  ({\rho'}^*L' -\rho^*L) - p {\rho'}^* X'_0 \right) \geq 0.\]
Note that $G'$ is the pullback of a $\Q$-Cartier $\Q$-divisor on $X'$, since
\[G'
\sim_{\Q}
m\rho^* L  +\left( m  ({\rho'}^*L' -\rho^*L) - p {\rho'}^* X'_0 \right)
 \sim_{\Q} {\rho'}^*( mL' - p X'_0).\]
Therefore, 
$G'$ is effective if and only if 
$\rho'_* G'$ is effective. 

To understand whether or not $\rho'_* G'$  is effective, 
observe
\[
{\rho'}^*L' -\rho^*L =  r\LARGE( \left(K_Y- \rho^*(K_X+D) \right) -\left(K_Y-\rho'^*(K_{X'}+D') \right) \LARGE) \]
and, hence,
\[
{\rho'}_*  \left( m({\rho'}^*L' -\rho^*L) - p{ \rho'}^*X'_0 \right)
= 
m \left( r (D' + a X'_0) -r D' \right)  - p X'_0 = (mra-p) X'_0
\]
Therefore, $ \rho'_* G' $ is effective if and only if $ {\rm coeff}_{\widetilde{X}'_0} ( \rho^* G) + (mra-p) \geq 0$. 
\end{proof}

 \subsubsection{Proof of S-completeness}
  
We are now in position to prove Theorem \ref{t-scomplete}  as a  consequence of the results in Sections  \ref{ss-spolar} and \ref{s-vanishingfilt}. %S-completness for polarized varieties}}.

\begin{proof}[Proof of Theorem \ref{t-scomplete}]
Following Section \ref{ss-spolar}%S-completness for polarized varieties
, we consider the $\Z$-graded $R[s,t]/(st-\pi)$-algebra 
\begin{equation}\label{e-ReesVV'}
 \underset{m \in \N}{\bigoplus} \underset{p \in \Z}{\bigoplus} \left( V_m \cap \pi^p V'_m \right)  t^{-p}.
 \end{equation}
Note that this algebra equals $\underset{m}{\bigoplus} \underset{p}{\bigoplus} \left( \cF^{p-mra}V_m \right) t^{-p}$  by Proposition \ref{p-comparefilt}
and its restriction to $0\in \Spec(R[s,t]/(st-\pi)$ is isomorphic to $\underset{m}{\bigoplus} \underset{p}{\bigoplus} \gr_{\cF}^{p-mra} V_{\kappa,m}$
by Equation \eqref{e-fiberover0}.
Since the latter $\kappa$-algebra is of finite type by Proposition \ref{p-finitegen} (1), 
Lemma \ref{L:filling_conditions} implies that the
the $R[s,t]/(st-\pi)$-algebra \eqref{e-ReesVV'} is  finite type.

Set $ \tilde{\cX} : = {\underline{ \Proj}}_{R[s,t]/(st-\pi)} \Big( 
 \underset{m}{\bigoplus} \underset{p}{\bigoplus} ( V_m \cap \pi^p V'_m)t^{-p}\Big)
$
and write $\tilde{\cD}$ for the component-wise closure of ${D \times( \mathbb{A}^{1} \setminus 0) }$
under the embedding 
${X  \times ( \mathbb{A}^{1} \setminus 0) 
\simeq 
\tilde{\cX}_{ t\neq 0  } 
\hookrightarrow
\tilde{\cX}}$.
The grading with respect to $p$ gives a $\mathbb{G}_m$-action on $\tilde{\cX}$ that fixes $\tilde{{\cD}}$.

We claim that  $\tilde{\cX} \to   \Spec\left( R[s,t]/(st-\pi) \right)$ has normal fibers, no component of $\tilde{\cD}$ contains a fiber, and  $K_{\tilde{\cX}}+\tilde{\cD}$ is $\Q$-Cartier.
The statement is clear away from the fiber over $0$. 
Next, note that  $\tilde{\cX}_{0}\simeq \Proj(  \bigoplus_m \bigoplus_{p}  \gr_{\cF}^{p-mra} V_{\kappa,m})$, 
which is the fiber over $0 \in \mathbb{A}^1_{\kappa}$ of the special test configuration in Proposition \ref{p-finitegen}. 
Hence, $\tilde{\cX}_{0}$ is normal.

To see $\tilde{\cX}_{0} \not\subset \Supp (\tilde{\cD})$, fix a prime divisor $B$ in the support of $D$ and write  $\tilde{\cB}$ for the closure of $B\times (\mathbb{A}^1 \setminus 0)$ in $\tilde{\cX}$. 
If $I_B \subseteq \bigoplus_m V_m$ is the homogenous ideal  defining $B$, then 
$\tilde{\cB}$ is defined by the homogenous ideal 
\[
\bigoplus_m \bigoplus_p ( I_B \cap  \cF^{p-mra} V_m ) t^{-p} \subseteq
\bigoplus_m \bigoplus_p   \left( \cF^{p-mra} V_m \right)  t^{-p} .
\]
Hence, the scheme theoretic fiber of $\tilde{\cB}$ over $0 \in \Spec(R[s,t]/(st-\pi))$
agrees with the vanishing of the ideal in \eqref{e-Idef}. 
Since the latter ideal defines a locus of codimension at least one in $\tilde{\cX}_0$ by Proposition \ref{p-cod1},  $\tilde{\cX}_0 \not\subset {\cB}$.

To see  $K_{\tilde{\cX}}+\tilde{\cD}$ is $\Q$-Cartier,  $K_{\tilde{\cX}}+\tilde{\cD}$ is $\Q$-Cartier
fix a  $\Q$-divisor $\tilde{\mathcal{L}}$ on $\tilde{\cX}$ such that $mr \tilde{\mathcal{L}}$ is in the linear equivalence class of  $\cO_{\tilde{\cX}} (m)$ for a positive integer $m$. By construction ${\mathcal{\tilde{L}}}|_{  s \neq 0  } \sim_{\Q}  (-K_{\tilde{\cX}}-\tilde{\cD})|_{  s \neq 0  }$. Therefore, 
$\tilde{\mathcal{L}} \sim_{\Q} -K_{\tilde{\cX}}-\tilde{\cD} + G$,
 for some $\Q$-divisor $G$ supported on  $\tilde{\cX}|_{ s = 0  }$. Since $\tilde{\cX}|_{ s = 0 } $ is an irreducible Cartier divisor, 
 $-K_{\tilde{\cX}}-\tilde{\cD}$ must be $\Q$-Cartier. 

Finally, note that   $(\tilde{\cX},\tilde{\cD})|_{s=0 }  \to \mathbb{A}^1_{\kappa}$ coincides with the 
 special test configuration in Proposition \ref{p-finitegen} by \eqref{e-restricts=0}. Therefore, $(\tilde{\cX}_{\overline{0}} , \tilde{\cD}_{\overline{0}})$ is a K-semistable log Fano pair. 
This implies $(\tilde{\cX},\tilde{\cD}) \to \Spec( R[s,t]/(st-\pi) )$ is a $\Q$-Gorenstein family of log Fano pairs and is the unique extension of  \eqref{e-punctured}  by Lemma \ref{L:filling_conditions}.
\end{proof} 

%\begin{rem}
%As in the introduction, let $\kkX^{\rm Kss}_{n,V}$ be the stack parameterizing $n$-dimensional  K-semistable $\bQ$-Fano varieties with volume $V$.  Theorem \ref{t-scomplete} immediately implies that $\kkX^{\rm Kss}_{n,V}$ is S-complete with respect to essentially of finite type DVRs.  We remind the reader that it is still unknown whether  $\kkX^{\rm Kss}_{n,V}$ is Artin since the property that K-semistability is an open condition is also unknown.
%
%Similarly, Theorem \ref{t-scomplete} implies that any suitably defined stack parameterizing K-semistable log Fano pairs is S-complete.  However, at the moment, it is not clear how one should define families of log Fano pairs (or more generally klt pairs) over non-reduced schemes. 
%\end{rem}

%%%
\section{Reductivity of the automorphism group}
In this section, we prove that if $(X,D)$ is a K-polystable log Fano pair, then the automorphism group 
\[
\Aut(X,D) : = \{ g \in \Aut(X) \, \vert\, g^*D = D\}
\]
is reductive (Theorem \ref{t-aut}).  

We note that this result would follow formally from results in the previous section if one could establish that a suitably defined stack parameterizing K-semistable log Fano pairs was represented by a finite type Artin stack.  Indeed, Theorem \ref{t-scomplete} would show that this stack is S-complete with respect to essentially of finite type DVRs  and therefore any closed point 
(i.e. a 
%\jarodchange{polystable object}
{K-polystable log Fano pair}) has reductive stabilizer (Remark \ref{R:reductive}). 
We will provide a direct alternative argument for the reductivity of $\Aut(X,D)$ inspired by the property of S-completeness.  Our argument has the advantage that it entirely avoids the language of stacks.

\subsection{Setup}
In this section, we fix a log Fano pair $(X,D)$ and write  $D= \sum_{i \in I} a_i D_i$ where the $D_i$ are distinct prime divisors. For each $a$ in the coefficient set $\{a_i \, \vert\, i \in I\}$, set $B_a  := \bigcup_{ a=a_i}D_i$.  Choose $r$ sufficiently divisible and large so that $\cL := \cO_X(-r(K_X+D))$ is a very ample line bundle.   

We will now equip $\Aut(X,D)$ with the structure of a linear algebraic group.  Since $\cL$ is very ample, $\Aut(X, \cL) := \{g \in \Aut(X) \, \mid \, g^* \cL \simeq \cL\}$ is a linear algebraic group as it is a closed subgroup of $\PGL(H^0(X, \cL))$.
For an element $g \in \Aut(X)$, observe that $g^*D = D$ if and only if $g^\ast(\cL) \simeq \cL$ and for all $a$ in the coefficient set, 
$g$ fixes $B_a$.  In other words, 
\[
\Aut(X,D) = \left\{ g \in \Aut(X, \cL) \, \left|  \,  \forall a, \, g(B_a) = B_a \right. \right\}
\]
As the conditions that $g(B_a) = B_a$ are closed conditions, this shows that $\Aut(X,D) \subset \Aut(X,\cL)$ is a closed subgroup.

\subsection{Isotrivial families of K-polystable log Fano pairs}
We begin by stating a special case of  Theorem \ref{t-scomplete} when the family is obtained by gluing two trivial families.  

Let $R$ be a DVR essentially of finite type over $k$ with
fraction field $K$ and residue field $\kappa$. 
%Fix a log Fano pair $(X,D)$ and write  $D= \sum_{i \in I} a_i D_i$ where the $D_i$ are distinct prime divisors.
%For each $a$ in the coefficient set $\{a_i \, \vert\, i \in I\}$, 
% set $B_a  := \bigcup_{ a=a_i}D_i$.
Fix a birational map $X_R\dashrightarrow X_R$ 
that induces an isomorphism $\alpha \co (X_K, D_K) \to (X_K,D_K)$.
As $ \Spec (R[s,t]/(st-\pi)) \setminus 0$ is the union of $\Spec(R[s]_s)$ and $\Spec(R[t]_t)$ along $\Spec(K[s]_s) = \bG_{m,K}$, we may glue the two trivial families $X_{R[s]_s} \to \Spec(R[s]_s)$ and $X_{R[t]_t} \to \Spec(R[t]_t)$ along the $\bG_m$-equivariant isomorphism induced by $\alpha$ to obtain a $\mathbb{G}_m$-equivariant
 $\Q$-Gorenstein family of log Fano pairs 
\begin{equation}\label{e-puncturedkps}
(\cX , \cD )
 \to 
 \Spec \left(R[s,t]/(st-\pi) \right) \setminus 0.
\end{equation}
Note that if we write $\cB_a$ for the closure of   $B_a\times \Spec(R[t]_t)$ under the inclusion
$ X_{R[t]_t} \hookrightarrow \cX$, then $\cD = \sum a \cB_a$. 
%\times  (\mathbb{A}^1 \setminus 0) \simeq  \cX_{t \neq 0}
\begin{prop}\label{p-scompKps}
If $(X,D)$ is K-polystable, 
then 
$$({\cX}, {\cD}) \to  \Spec\left(R[s,t]/(st-\pi)\right) \setminus 0$$
extends to a $\mathbb{G}_m$-equivariant $\Q$-Gorenstein  family of log Fano pairs 
\[
(\tilde{\cX}, \tilde{\cD}) \to  \Spec\left(R[s,t]/(st-\pi)\right)
.\]
with $(\tilde{\cX}_{ \overline{0}}, \tilde{\cD}_{ \overline{0}}) \simeq (X_{\overline{\kappa}},D_{ \overline{\kappa}})$. 
Furthermore, if we write $\tilde{\cD} = \sum a \tilde{\cB}_a$, where $\tilde{\cB}_a$ is the closure of $\cB_a$, 
then each $\tilde{\cB}_a$ is flat over $\Spec(R[s,t]/(st-\pi) )$ with pure fibers.
\end{prop}

By \emph{pure} fibers, we mean that the fibers are equidimensional and have no embedded components.

\begin{proof}
By Theorem \ref{t-scomplete}, the map in \eqref{e-puncturedkps} extends
to  a family  
$(\tilde{\cX}, \tilde{\cD})$ 
with K-semistable geometric fiber over $0 \in  \Spec\left(R[s,t]/(st-\pi)\right)$. 
Hence, the restriction $(\tilde{\cX}, \tilde{\cD})\vert_{s=0, \overline{\kappa}}$ is naturally 
a special test configuration of $(X_{\overline{\kappa}}, D_{\overline{\kappa}})$ with K-semistable geometric fiber over $0 \in \mathbb{A}^1_\kappa$.  
Since $(X,D)$ is K-polystable, this test configuration must be a product 
(i.e. it is isomorphic to $(X_{\mathbb{A}_{\overline{\kappa}}^1}, D_{\mathbb{A}_{\overline{\kappa}}^1})$).
Therefore, $(\tilde{\cX}_{{ \overline{0} }}, \tilde{\cD}_{{\overline{0} }}) \simeq (X_{\overline{\kappa}},D_{ \overline{\kappa}})$.

Next, fix $a$ in the coefficient set of $D$ and consider the divisor $\tilde{\cB}_a$. 
By \cite[Thm. 4.33]{Kol19}, there exists a locally closed decomposition $S=\sqcup S_i \to \Spec(R[s,t]/(st-\pi))$ such that a morphism of schemes $T\to  \Spec(R[s,t]/(st-\pi))$, with $T$ reduced, factors through $S$
if and only if $\tilde{\cB}_a \vert_T \to T$ is flat and has pure fibers. Such locally closed decomposition is unique by definition. 

Now, the loci $\{s=0\}$, $\{s \neq 0\}$ and $\{t \neq 0\}$ factor through $S$, since the divisorial restrictions $\tilde{\cB}_a|_{s=0}$, $\tilde{\cB}_a|_{s\neq 0}$  and $\tilde{\cB}_a|_{t\neq 0}$ are trivial families. Therefore, each of them factors through some locally closed set, denoted by $S_0$, $S_1$ and $S_2$. However, it then follows that $S_0=S_2$ as $\{s=0\} \cap \{t \neq 0\} \neq \emptyset$ and  $S_1=S_2$ as  $\{s \neq 0\}\cap \{t \neq 0\} \neq \emptyset$. 
Therefore, we have $S=S_0 = \Spec(R[s,t]/(st-\pi))$, so  $ \tilde{\cB}_a \to \Spec\left(R[s,t]/(st-\pi)\right)$ is flat with pure fibers.
\end{proof}

\subsection{Reductivity via Iwahori decompositions}
Throughout this section, let $R=k[[\pi]]$ and $K=k((\pi))$. 
Given a linear algebraic group $G$,
Iwahori's theorem (cf. \cite[p.52]{git}) states that if $G$ is reductive,
then for any element $g\in G(K)$, there exist
$a,b\in G(R)$ and a one-parameter subgroup $\lambda\in {\rm Hom}(\mathbb{G}_m, G)$ such that
$g=a\cdot \lambda\vert_K\cdot b$,
where $\lambda\vert_K$ denotes the composition $\Spec(K) \to \mathbb{G}_m =\Spec(k[\pi]_{\pi}) \overset{\lambda}{\to} G$.  If we let  $\Lambda_G$ denote  the set of $K$-points induced by one-parameter subgroups of $G$, then Iwahori's theorem states that if $G$ is reductive, then
$$
G(K)=G(R)\Lambda_GG(R).
$$
The following  argument, which states that the converse also holds, was  communicated to us by Jun Yu.
 See \cite{AHH19} for a proof using Artin stacks and S-completeness.
 
\begin{prop}\label{l-iwahori}
Let $G$ be a linear algebraic group. If $G(K) = G(R)\Lambda_GG(R)$,
then $G$ is reductive. 
\end{prop}
\begin{proof}
%Arguing by a contradiction, assume $G$ is not reductive.
Write $G = G_u \rtimes G_s$ for the Levi decomposition of $G$. That means, $G_u$ is the
unipotent radical of $G$, which is a (connected) unipotent group over $k$, and $G_s$ is
a reductive group over $k$; the map $(x; y) \to xy$ where $(x\in G_u, y\in G_s)$ gives a
bijection $G_u\times G_s \to G$.

For any one-parameter subgroup $\lambda\colon \mathbb{G}_m \to G$ defined over $k$, the image of $\lambda$ consists of semisimple elements. 
Thus, it is contained in a conjugate of $G_s$. That means, there exists $g \in G(k)$ such that ${\rm Ad}(g)(\lambda)$ has image lying in $G_s$, or in other words 
${\rm Ad}(g)\cdot \lambda \in \Lambda_{G_s}$. Therefore,
$$\Lambda_G = {\rm Ad}(G(k))(\Lambda_{G_s}) \subset G(k)\Lambda_{G_s}G(k).$$
Since $G = G_u \rtimes G_s$, we get
$$ G(R) = G_u(R) \rtimes G_s(R) = G_u(R)G_s(R) = G_s(R)G_u(R).$$
Combining the above, we get
\begin{eqnarray*}
G(R)\Lambda_GG(R)&\subset&G(R)G(k)\Lambda_{G_s}G(k)G(R) \\
 &=&G(R)\Lambda_{G_s}G(R)\\
&=& G_u(R)G_s(R)\Lambda_{G_s}G_s(R)G_u(R)\\
&=& G_u(R)G_s(K)G_u(R),
\end{eqnarray*}
where we used Iwahori's theorem to $G_s$ in the last equality.
Thus, $G(R)\Lambda_GG(R) = G_u(R)G_s(K)G_u(R).$

Suppose 
$$G(K) = G(R)\Lambda_GG(R) = G_u(R)G_s(K)G_u(R).$$ Then,
$$G_u(K)=\big(G_u(R)G_s(K)G_u(R)\big)\cap G_u(K)=G_u(R)\big(G_s(K)\cap G_u(K)\big)G_u(R)=G_u(R).$$
Since $G_u$ is a connected unipotent group, we have   $G_u \cong \mathbb{A}^n$ as a variety over $k$, where $n = \dim G_u$. Then,
$G_u(K) = G_u(R)$ implies that $K^n = R^n$, which %is absurd
in turn implies that $n=0$ and that $G$ is reductive.
\end{proof}

The following lemma will allow us to work with essentially finite type DVRs when checking 
that the hypotheses of Proposition \ref{l-iwahori} are satisfied. 

\begin{lem} \label{l-artin-approx}
If $G$ is a linear algebraic group, then for any $g\in G(K)$,  there is an algebraic point $g_0$ such that $g \cdot g_0^{-1} \in G(R)$, where algebraic means that $g_0 \in G(k(C))$ for the function field $k(C)$ of a smooth curve over $k$ embedded in $K$ via a dominant morphism $\Spec(R) \to C$.
\end{lem}
\begin{proof}
Fix an embedding $G \subset {\rm GL}_m$ for some $m>0$ 
and $N\gg0 $ so that $\pi^N\cdot g^{-1}\in 
%\jarodchange{${\rm GL}_m(R)$}
M_{m\times m}(R)$.
 By Artin approximation (\cite{Art69}), we can find an algebraic point $g_0\in G(K)$ such that $g-g_0\in \pi^{N+1}\cdot {M}_{m\times m}(R)$. 
Since
$g g_0^{-1}=\left(1- \frac{g-g_0}{g}\right)^{-1}$
and $\frac{g-g_0}{g} \in \pi  \cdot  {M}_{m\times m}(R)$, we know 
$$g\cdot g_0^{-1}=1+\sum_{i=1}^{\infty}
\left( \frac{ g_0 -g}{g} \right)^i \in {\rm GL}_m(R) \cap G(K) = G(R).$$
\end{proof}

\subsection{Proof of reductivity}

Theorem \ref{t-aut} is an immediate consequence of Lemma \ref{l-iwahori} and the following proposition.

\begin{prop} \label{P:proof-reductivity}
If $(X,D)$ is a K-polystable log Fano pair, then $G:=\Aut(X,D)$ satisfies  $G(K)=G(R)\Lambda_GG(R)$. 
\end{prop}

\begin{proof}
Set $R = k[[\pi]]$ and $K = k((\pi))$.  By Lemma \ref{l-artin-approx}, it suffices to show that all algebraic points of $G(K)$ are contained in $G(R)\Lambda_GG(R)$. 
To proceed, fix a smooth pointed curve $x \in C$ with local ring $R_0 : = \cO_{C,x}$, function field $K_0: = {\rm Frac}( \cO_{C,x})$, and an extension of DVRs $R_0 \subset R$.
 We will show that  if $g\in G(K_0)$, then $g \vert_K \in G(R)\Lambda_GG(R)$.

Consider the isomorphism  $(X_{K_0}, D_{K_0}) \to (X_{K_0}, D_{K_0})$ of log Fano pairs induced 
by $g$. 
This data gives a $\mathbb{G}_m$-equivariant $\Q$-Gorenstein family of log Fano pairs 
\[
(\cX, \cD) \to \Spec \left(R_0[s,t]/(st-\pi) \right) \setminus 0 
\]
and we may write $\cD = \sum a \cB_a$.  
By Proposition \ref{p-scompKps},  the above family extends to a  $\mathbb{G}_m$-equivariant $\Q$-Gorenstein family of log Fano pairs  $( \tilde{\cX}, \tilde{\cD}) \to \Spec(R_0[s,t]/(st-\pi))$
such that $( \tilde{\cX}_{0}, \tilde{\cD}_{0}) \simeq (X,D)$.  Moreover, $\tilde{\cD} = \sum a \tilde{\cB}_a$ where each $\tilde{\cB}_a$ is flat over $\Spec(R_0[s,t]/(st-\pi) )$ with  pure fibers.
The $\mathbb{G}_m$-action on the fiber  $( \tilde{\cX}_{0}, \tilde{\cD}_{0})$ induces a 1-parameter subgroup $\lambda: \mathbb{G}_m \to G$.

Replace $( \tilde{\cX}, \tilde{\cD})$ with its base change by $R$ to get a family over $\cS := \Spec \left(R[s,t]/(st-\pi) \right)$. We will show that there is a $\bG_m$-equivariant isomorphism $( \tilde{\cX}, \tilde{\cD}) \cong (X_{\cS}, D_{\cS})$, where $\bG_m$ acts on $(X_{\cS}, D_{\cS}) = (X \times \cS, D \times \cS)$ diagonally, via $\lambda$ on the left factor and the standard action $\cS$. As every geometric fiber of the family $( \tilde{\cX}, \tilde{\cD}) \to \Spec(R)$ is isomorphic to the base change of $(X,D)$ and since each $\tilde{\cB}_a$ is flat over $\cS$, the scheme
\[
\cI := \Isom_{\cS}((\tilde{\cX} , \tilde{\cD}) , (X_{\cS}, D_{\cS}))
\]
parameterizing isomorphisms is a $G$-torsor over $\cS$ (c.f. \cite[Lemma 2.3.2]{dejong-starr}). For any test scheme $T$, a $T$-point of $\cI$ consists of a point $p \in \cS(T)$ along with an isomorphism $\phi : (\tilde{\cX}_p, \tilde{\cD}_p) \cong (X_T,D_T)$ of families over $T$. The $\bG_m$-action on both pairs gives a $\bG_m$-action on $\cI$, where for any test scheme $T$, a $T$-point $t \in \bG_m(T)$ acts on $\cI(T)$ by
\[
t \cdot(p,\phi) = \left(t \cdot p, \lambda(t) \cdot \phi(t^{-1} \cdot (-)) : \tilde{\cX}_{t\cdot p} \cong X_{T}\right).
\]
Note that the projection $\cI \to \cS$ is $\bG_m$-equivariant, and a $\bG_m$-equivariant section of this morphism classifies a $\bG_m$-equivariant isomorphism of families $(\tilde{\cX} , \tilde{\cD}) \cong (X_{\cS}, D_{\cS})$ over $\cS$.

%the fact that $\bG_m$ is linearly reductive then implies that the fixed locus $\cI^{\bG_m}$, parameterizing $\bG_m$-equivariant isomorphisms, is also smooth over $\cS$ (c.f. \cite[XII.9.6]{sga3}). 

The projection $\cI \to \cS$ is smooth because it is a principal $G$-bundle and $G$ is smooth. Let $\cS_n$ be the $n$th nilpotent thickening of $0 \in \cS$.  By construction, we have a $\bG_m$-equivariant section $s_0 \co \cS_0 \to \cI$. By the formal lifting criteria for smoothness, $s_0$ extends to a compatible family of  $\bG_m$-equivariant sections $s_n \co \cS_n \to \cI$. We claim that the sections $s_n$ algebraize to a $\bG_m$-equivariant section $s \co \cS \to \cI$.   The $\bG_m$-actions induce $\bZ$-gradings  $\Gamma(\cO_\cS) = \bigoplus_d \Gamma(\cO_\cS)_d$,  $\Gamma(\cO_{\cS_n}) = \bigoplus_d \Gamma(\cO_{\cS_n})_d$ and $\Gamma(\cO_\cI) = \bigoplus_d \Gamma(\cO_\cI)_d$.
To prove the existence of the desired section $s \co \cS \to \cI$, it suffices to verify the existence of a graded homomorphism $\Gamma(\cO_\cI) \to \Gamma(\cO_\cS)$ extending the given homomorphisms $\Gamma(\cO_\cI) \to \Gamma(\cO_{\cS_n})$. To see this, observe that for each $d$, the compatible maps $\Gamma(\cO_\cI)_d \to \Gamma(\cO_{\cS_n})_d$ extend to a map $\Gamma(\cO_\cI)_d \to \ilim_n \Gamma(\cO_{\cS_n})_d$. The latter $R$-module can be explicitly computed to be isomorphic to $\Gamma(\cO_\cS)_d$ since $R$ is complete.

To conclude, let $\phi : (\tilde{\cX},\tilde{\cD}) \cong (X_\cS,D_\cS)$ be the $\bG_m$-equivariant isomorphism constructed in the previous paragraph. Restricting to $\cS \setminus 0$ and quotienting by the $\bG_m$-action, $\phi$ gives an isomorphism between two families over $\Spec(R) \cup_{\Spec(K)} \Spec(R)$. Each family was obtained by gluing two copies of the trivial family along an isomorphism over $\Spec(K)$, the first family corresponding to $g \in G(K)$ and the second to $\lambda(\pi) \in G(K)$. Thus $\phi\vert_{(\cS \setminus 0)/\bG_m}$ corresponds to a pair $a,b \in G(R)$ such that $a \cdot g = \lambda(\pi) \cdot b$, and hence $g = a^{-1} \cdot  \lambda(\pi) \cdot b \in G(R)\Lambda_GG(R)$.

\end{proof}

\section{$\Theta$-reductivity}\label{s-theta}

%We will discuss S-completeness and $\Theta$-reductivity in the context of polarized varieties.  

In this section, we will carry out an analysis similar to that in Section \ref{s-Scomplete}  for $\Theta$-reductivity.

\subsection{$\Theta$-reductivity for coherent sheaves}  
Let $R$ be a DVR with fraction field $K$, residue field $\kappa$ and uniformizing parameter $\pi$.  Recall that $\Theta = [\bA^1/\bG_m]$ and that $\Theta_R = \Theta \times \Spec(R) =  [\Spec(R[x]) / \bG_m]$, where $x$ has weight $-1$.  The reader may wish to refer to the schematic picture \eqref{E:ThetaR-schematic}  of  $\Theta_R$.   In this subsection, we establish $\Theta$-reductivity for the stack parameterizing coherent sheaves on $\Spec(k)$ or, in other words, that every flat and coherent sheaf on $\Theta_R \setminus 0$ extends uniquely to a flat and coherent sheaf on $\Theta_R$.

A quasi-coherent sheaf $F$ on $\Theta_R$ corresponds to a $\bG_m$-equivariant quasi-coherent sheaf on $\Spec(R[x])$ or, in other words, 
a $\bZ$-graded $R[x]$-module $\bigoplus_{p \in \bZ} F_p$; this in turn corresponds to a diagram $ \cdots \xrightarrow{x} F_{p+1} \xrightarrow{x} F_p \xrightarrow{x} F_{p-1} \xrightarrow{x} \cdots$ of $R$-modules.   The restriction of $F$ to $\Spec(R) \xhookrightarrow{x \neq 0} \Theta_R$ is the $R$-module $\colim F_p$ and the restriction to $\Theta_{\kappa} \xhookrightarrow{\pi= 0} \Theta_R$ is the $\bZ$-graded $\kappa$-module $\bigoplus_{p \in \bZ} F_p/\pi F_{p}$.  Moreover, $F$ is flat and coherent  over $\Theta_R$ if and only if each $F_p$ is flat and coherent over $R$, the maps $x \co F_{p+1} \to F_{p}$ are injective, each $F_p/F_{p+1}$ is flat, $F_p = 0$ for $p \gg 0$, and $F_p$ stabilize for $p \ll 0$.
%\footnote{ \color{blue}  
%Maybe we need additional assumptions (i.e. the restriction to $\Spec(R)$ is flat) or I am misinterpreting the statement. 
%If we start with a non-flat $R$-module and give it a trivial filtration, then the maps given by $x$ will be injective. (Also, there is a similar statement in 3.1.) -HB
%\jarod{You are right.  I added the appropriate conditions here and in Section 3.}}

We will compute the pushforward  along the open immersion $j \co \Theta_R \setminus 0 \hookrightarrow \Theta_R$.  Denote the open immersions by 
$$\begin{aligned}
j_x \co \Spec(R) \xhookrightarrow{x \neq 0}  \Theta_R, \,
j_{\pi} \co \Theta_K \xhookrightarrow{\pi \neq 0} \Theta_R \mbox{ \, and \,}
j_{x\pi} \co \Spec(K) \xhookrightarrow{x\pi \neq 0}  \Theta_R.
\end{aligned}$$
  Let $\cE$ be a flat coherent sheaf on $\Theta_R \setminus 0$; this corresponds to 	a free $R$-module $E$ of finite rank and a $\bZ$-filtration $\cG^{\bullet}E_K \co \cdots \subset \cG^{p+1}E_{K}  \subset \cG^{p}E_{K}  \subset \cdots $ of $E_K$.   Then 
  %\jarodchange{$j_\ast \cE = \ker\left( (j_x)_\ast E \oplus (j_\pi)_\ast \cG^{\bullet} E_{K} \to (j_{x \pi})_\ast(E_K) \right).$}
 {$j_\ast \cE =  (j_x)_\ast E \cap (j_\pi)_\ast \cG^{\bullet} E_{K} \subset (j_{x \pi})_\ast E_K.$}
As morphisms of graded $R[x]$-modules, $j_x$ and $j_{\pi}$ correspond to the inclusions $R[x] \subset R[x]_x$ and $R[x] \subset K[x]$, and $j_{x\pi}$ corresponds to $R[x] \subset K[x]_x$.  We 
 compute that
$$\begin{aligned}
	(j_{x \pi})_\ast E_K	& \cong K[x]_x \otimes_R E_K \cong \bigoplus_{p \in \bZ} E_K x^{-p}, \\
	(j_{x})_\ast E 		& \cong E \otimes_R R[x]_x \cong \bigoplus_{p \in \bZ} E x^{-p}  \subset (j_{x \pi})_\ast E_K,  \\
	(j_{\pi})_\ast \cG^{\bullet} E_{K} & \cong \bigoplus_{p \in \bZ} (\cG^p E_{K}) x^{-p} \subset (j_{x\pi})_\ast E_K
\end{aligned}$$
Therefore
\begin{equation} \label{E:push-forward2}
j_\ast \cE \cong \bigoplus_{p \in \bZ} \big( E \cap \cG^p E_{K} \big) x^{-p} \subset \bigoplus_{p\in \bZ} E_K x^{-p}.
\end{equation}
The sheaf $j_\ast \cE $ is flat and coherent over $\Theta_R$, and is given by the filtration $\cG^pE := E \cap \cG^p E_{K}$ of $E$.

\subsection{$\Theta$-reductivity for polarized families}\label{ss-tpolar} 
A polarized family $(\cX, \cL)$ over $\Theta_R \setminus 0$ corresponds to a polarized family $(X,L)$ over $\Spec(R)$ and a polarized family $(\cX_K,\cL_K)$ over $\Theta_K$ together with an isomorphism of $(X_K, L_K)$ with the fiber of $(\cX_K, \cL_K)$ over $1$.  The polarized family $(\cX_K,\cL_K)$ over $\Theta_K$ corresponds to a test configuration over $\bA^1_K$.   %\cdots \subset \cG^p H^0(\cX_k,\oh_{\cX_K}(m\cL_K)) \subset \cG^{p+1} H^0(\cX_k,\oh_{\cX_K}(m\cL_K)) \subset \cdots$.

% \to (Y_1, N_1)$.  %For each $m$, define the $R$-module $R_d := H^0(X, L^d)$ and the $K$-module $S_d := H^0(Y, N^d)$ which is endowed with a $\bZ$-filtration $S_{p, \bullet} \co \cdots S_{p,m} \subset S_{p,m+1} \subset \cdots $.  
Consider the composition $\cX \xrightarrow{q} \Theta_R \setminus 0 \xrightarrow{j} \Theta_R$. 
\begin{condition}[Finite Generation Condition]\label{a-finitegenerate2}
The  $\oh_{\Theta_R}$-algebra   $\bigoplus_{m \ge 0} j_* q_* \oh_{\cX}(m \cL)$  is finitely generated.
\end{condition}

If Condition \ref{a-finitegenerate2} holds, then
$$\tilde{\cX} := \uProj_{\, \Theta_R} \bigoplus_{m \ge 0} j_* q_* \oh_{\cX}(m\cL),$$
is a flat family of polarized schemes over $\Theta_R$. % with polarization $\oh_{\tilde{\cX}}(1)$.

For each $m \ge 0$, set $V_m :=H^0(X, \oh_X(mL)) $.   For each $m \ge 0$, the vector space $V_{K,m}:=H^0(\cX_K,\oh_{\cX_K}(m\cL_K))$ inherits a $\bZ$-filtration $\cG^{\bullet} V_{K,m}$. Equation \eqref{E:push-forward2} yields
$$\begin{aligned}
j_* q_* \oh_{\cX}(m\cL) &\cong \bigoplus_{p \in \bZ} \big(V_m \cap \cG^p V_{K,m} \big)x^{-p}  \subset \bigoplus_{p \in \bZ} V_{K,m} x^{-p}.
\end{aligned}$$

If we set $\cG^pV_{m} = V_m \cap \cG^p V_{K,m}$, then the direct sum $\bigoplus_{p,m} \cG^pV_{m}$ is a bigraded $R[x]$-module, where multiplication by $x$ is given by the inclusions $\cG^pV_{m} \to \cG^{p-1}V_{m}$. 
  The grading in $m$ defines a projective scheme $\cP = \uProj_{\Spec(R[x])} \bigoplus_{p,m} \cG^pV_{m}$ and the grading in $p$ gives an action of $\bG_m$ on $\cP$ and a linearization of $\oh_{\cP}(1)$.  
  Observe that $(\tilde{\cX}, \oh_{\tilde{\cX}}(1))= ([\cP / \bG_m], \oh_{\cP}(1))$.

%%%%%%%%%%%%%%%%%%%%%%%%%%%%%%%%%%%

\subsection{$\Theta$-reductivity for K-semistable log Fano pairs}
In this section, we will verify that $\kkX_{V,n}^{{\rm Kss}}$ satisfies the valuative criterion for 
$\Theta$-reductivity over any essentially finite type DVR.
The result follows from modifying an argument in \cite[Sect. 3]{LWX18}.

\medskip

Fix the following notation: Let $R$ be a DVR essentially of finite type over $k$ 
with fraction field $K$ and residue field $\kappa$. We will write $x$ for the parameter of $\mathbb{A}^1$. 
To avoid confusion, we write $0_K \in \mathbb{A}^1_K$ for the closed point defined by the vanishing of $x$ 
and $0 \in \mathbb{A}^1_R$ for the one defined
by the vanishing of $x$ and a uniformizing parameter $\pi \in R$.

Fix  a $\Q$-Gorenstein family of log Fano pairs $(X,D) \to \Spec(R)$ 
and a special test configuration $(\cX_K,\cD_K) \to \mathbb{A}^1_K$ of $(X_K,D_K)$.  
Following Section \ref{ss-tpolar}%Jarod's section on \theta reductivity for polarized pairs
, this data gives a $\mathbb{G}_m$-equivariant $\Q$-Gorenstein family of log Fano pairs 
\[
(\cX , \cD ) \to   \mathbb{A}^1_R  \,\setminus\,  0\, 
.\]

\begin{thm}\label{t-Thetared}
If the geometric fibers of $(X,D) \to \Spec(R)$ and $(\cX_K, \cD_K) \to \mathbb{A}^1_K$ are K-semistable, 
then $(\cX , \cD ) \to \mathbb{A}^1_R \setminus\, 0\,$ 
extends  uniquely to a $\mathbb{G}_m$-equivariant $\Q$-Gorenstein family of log Fano pairs 
\[
(\tilde{\cX}, \tilde{\cD}) \to \mathbb{A}^1_R  
.\]
Furthermore, the geometric fiber over $0$ is  K-semistable.
\end{thm}

 Throughout the proof, we will use notation similar to that in Section \ref{s-vanishingfilt}. 
Specifically, fix a positive integer  $r$ such that $L:=-r(K_X+D) $ is a Cartier divisor. 
Let
$V:=  \bigoplus_m V_m$
denote the section ring of $X$ with respect to $L$. 
Recall that each $V_m$ is a flat $R$-module 
and the restrictions of $V$ to $\Spec(K)$ and $\Spec(\kappa)$, which we denote by 
$ V_{K} := \bigoplus_m V_{K,m} $ and $V_\kappa := \bigoplus_m V_{\kappa,m}$,
are isomorphic to the section rings of $L_K$ and $L_{\kappa}$, respectively.

\subsubsection{Extending filtrations defined by a divisor}
Let $E_K$ be a divisor over $X_K$ and write $A: = A_{X_K,D_K}(E_K)$.  
Setting
$$\cF_K^p V_{K,m}: = \{ f \in V_{K,m} \, \vert \, \ord_{E_K}(f) \geq p \},$$
for each $p \in \Z$ and $m \in \N$, 
gives a filtration of $V_{K}$.
The filtration $\cF^{\bullet}_{K}$ of $V_{K,m}$ extends to a filtration $\cF^\bullet$ of $V_{m}$ by subbundles by setting
\[
\cF^{p} V_m : = \cF_{K}^p V_{K,m} \cap V_m .\]
Note that
$
 \bigoplus_{m } \bigoplus_{p } \left( \cF^p V_m \right) x^{-p} 
$
is a graded $R[x]$-algebra. 

If the above algebra is finitely generated, we  set
 $\widetilde{\cX}:= \underline{\Proj}_{\mathbb{A}^1_R} \left(  \bigoplus_{m} \bigoplus_{p} \left( \cF^p V_m \right) x^{-p} \right)$.
Since
$$ \bigoplus_{m \in \N} \bigoplus_{p \in \Z} \left( \cF^p V_m \right) x^{-p} 
\otimes_{R[x]} R[x,x^{-1}] \simeq  V \bigotimes_R R[x,x^{-1}]
 $$
there is an isomorphism $X \times  (\mathbb{A}^1 \setminus 0 ) \simeq \widetilde{\cX}_{x\neq 0 }$.
We write $\widetilde{\cD}$ for the closure of $D \times (\mathbb{A}^1 \setminus 0)$ under the previous  embedding 
$X \times (\mathbb{A}^1 \setminus 0 ) \hookrightarrow \widetilde{\cX}$.

%\dan{Comment: this seems to be repeated from the discussion in 5.2, except that here we are looking at the order filtration coming from $E_K$. Do you think we should move this there and remove the more stacky discussion?}

\begin{prop}\label{p-extendE}
If the geometric fibers of $(X,D) \to \Spec(R)$ are K-semistable
and $\beta_{X_K,D_K}(E_K) =0$, then:
\begin{enumerate}
\item The $R[x]$-algebra $\bigoplus_{m} \bigoplus_{p} \left( \cF^p V_m \right) x^{-p} $ is of finite type. 
\item The induced family $(\cX,\cD) \to \mathbb{A}^1_R$ is a $\Q$-Gorenstein family of log Fano pairs and $(\cX_{\overline{0}},\cD_{\overline{0}})$ is  K-semistable.
\end{enumerate}
\end{prop}

The proof is a modification of an argument in \cite[Sect. 3]{LWX18}. Similar arguments are also used to prove the main theorems in \cite{BX18}. Throughout, we will use notation and background material from \cite[Sect. 2]{BX18} on valuations, log canonical thresholds, and the normalized volume function.

\begin{proof}
Let $(Y,\Gamma) \to \Spec(R)$ denote the relative cone over $(X,D) \to \Spec(R)$ with respect to the polarization $L$.
%and $\sigma: C \to Y$ the section of cone points.  
Hence, $Y = \underline{\Spec}_R( V)$ and $\Gamma$ is defined via pulling back $D$. 
Note that $(Y_K,\Gamma_K)$ and $(Y_\kappa,\Gamma_\kappa)$ 
are the cones over $(X_K,D_K)$ and $(X_{\kappa}, D_{\kappa})$. 

Following \cite[Sect. 2.5.1]{BX18}, 
the divisor $E_K$ over $X_K$ induces a ray of quasimonomial valuations 
\[
\{v _t  \, \vert \, t\in [0,+\infty) \}  \subset \Val_{Y_K}
\]
satisfying
$$
A_{Y_K, \Gamma_K}(v_t) = 1/r + t A
\mbox{\ \ and\ \ } 
\fa_{p}(v_t)  =  \bigoplus_m \cF_K^{(p-m)/t} V_{K,m}.
$$

For each $q\in \N$, there is a divisor $E_{K,q}$ over $Y_K$ such that $q\cdot v_{1/q} = \ord_{E_{K,q}}$. 
The divisor  $E_{K,q}$ over $Y_K$ extends to a divisor $E_{q}$ over $Y$. 
Note that 
$
A_{Y, \Gamma} (E_q) =
A_{Y_K, \Gamma_K}(E_{K,q})$
and 
\begin{equation}\label{e-idealE_k}
\fa_{p}(\ord_{ E_{q}}) =  \bigoplus_{m \in \N} \cF^{p-mq} V_{m}.
\end{equation}
To see Equation \eqref{e-idealE_k} holds, observe that the order of vanishing of  $f\in \cO_{Y}$ along $E_{q}$ equals the order of vanishing of 
 $f \cdot \cO_{Y_K}$ along $E_{K,q}$.
Hence, the statement follows from the definition of $\cF$ and the formula
$
\fa_{p}(\ord_{E_{K,q}}) =  \bigoplus_m \cF_{ {K}}^{p-mq} V_{K,m}$.\medskip

\noindent \emph{ Claim 1}. The following holds:
\[
\lim_{q \to \infty}  \left( A_{Y, \Gamma+ Y_\kappa}( E_q )  - \lct(Y, \Gamma+ Y_\kappa; \fa_\bullet ( \ord_{E_q} ))  \right)
= 0.\]

To prove this claim, for each positive integer $q$, consider the graded sequence of ideals on $Y_\kappa$ given by 
\[
\fb_{q ,\bullet} : = \fa_\bullet(\ord_{E_q}) \cdot \cO_{Y_\kappa}.\]
Note that 
$\fa_{\bullet}(\ord_{E_{K,q}})  = 
\fa_\bullet(\ord_{E_q}) \cdot \cO_{Y_K}$
by \eqref{e-idealE_k}.
Therefore, the lower semicontinuity of the log canonical threshold and  \cite[Eq. (3)]{BX18} imply
\begin{equation}\label{e-lctAineq}
\lct\left(Y_\kappa, \Gamma_\kappa; \fb_{q,\bullet} \right) 
\leq 
\lct\left(Y_K,\Gamma_K; \fa_{\bullet} (\ord_{E_{K,q} }) \right) 
\leq 
A_{Y_K, \Gamma_K}( E_{K,q}) 
\end{equation}
Additionally,
\small{
\begin{equation}\label{e-multequal}
\mult(
\fb_{q,\bullet} ) 
= 
\lim_{p \to \infty}  \frac{ \dim_{\kappa} \left( \cO_{Y_\kappa} /  \fb_{q,p}\right) }{p^{n+1}/ (n+1)! }
=
\lim_{p \to \infty}  \frac{ \dim_{K} \left( \cO_{Y_K} /  \fa_{p}(\ord_{E_{K,q}}) \right)}{p^{n+1}/ (n+1)! } 
 = \mult(  \fa_\bullet(\ord_{E_{K,q}}) ),
 \end{equation}
 }
 where the left and right equalities is the formula for multiplicity in \cite[Thm 3.8]{LM09}
 and the center equality follows from \eqref{e-idealE_k} and the fact that each $\cF^p V_m \subset V_m$ is a subbundle.

We aim to show the inequalities:
\small{
 \begin{equation}\label{e-keyineq}
\frac{Q}{r}  \leq 
\lct( Y_\kappa ,\Gamma_\kappa;
\fb_{q,\bullet}  )^{n+1}
 \mult(
\fb_{q,\bullet} )  
\leq 
A_{Y_K, \Gamma_K}(E_{K,q})^{n+1}
 \mult(
\fa_{\bullet} (\ord_{E_{K,q}})) 
\leq \frac{Q}{r} + O\left( \frac{1}{q^2}\right)
,\end{equation}
}
with $Q: = (-K_{X_K}-D_K)^n= (-K_{X_\kappa}-D_\kappa)^n$.
The first inequality follows from \cite[Thm 7]{Liu18} and the assumption  that $(X_{\overline{\kappa}},D_{\overline{\kappa}})$ is K-semistable and the second  from \eqref{e-lctAineq} and \eqref{e-multequal}. 
For the remaining equality, Li's derivative formula (for example, see \cite[Prop. 2.12]{BX18}) gives
\[
\frac{d \,  \hvol(v_t) }{dt}\bigg \lvert_{t=0^+}  = (n+1) \beta_{X_K,D_K}(E_K).\]
%Observe that $\Fut( \cX_K, \cD_K)=0$.
%Indeed, $\Fut(\cX_K, \cD_K)=\Fut (\cX_{\overline{K}}, \cD_{ \overline{K}})$
%and the latter is the same as the Futaki invariant associated to the  $\mathbb{G}_m$-action on $(\cX_{\overline{K}}, \cD_{ \overline{K}})_0$. But, since the Futaki invariants associated to the $\mathbb{G}_m$-action and its inverse add to zero \cite[2.23]{LWX18} and $(\cX_{\overline{K}}, \cD_{ \overline{K}})_0$ is K-semistable, they must both be zero.
Since the latter is zero, a Taylor expansion implies
\[
\hvol (v_{1/q}) = \hvol(v_0) + O\left( \frac{1}{q^2} \right) = \frac{Q}{r} + O\left( \frac{1}{q^2} \right).\]
Using that $\hvol$ is scaling invariant, we observe
\[
\hvol(v_{1/q}) = \hvol( \ord_{E_{K,q}}) := A_{Y_K, \Gamma_K}(E_{K,q})^{n+1} \mult( \fa_\bullet(\ord_{E_{K,q}}))
 \] and \eqref{e-keyineq} follows.

Comparing  \eqref{e-multequal} and \eqref{e-keyineq}, we see
\[
\frac{1}{1 + O\left( \frac{1}{q^2} \right)}  
\leq 
\left(
\frac{ 
\lct( Y_\kappa,\Gamma_\kappa;
\fb_{q,\bullet} )}{A_{Y_K, \Gamma_K}(E_{K,q})}
\right)^{n+1} \leq 1
.\]
Since
$
A_{Y_K, \Gamma_K}(E_{K,q}) = A_{Y, \Gamma}( E_q) =  A_{Y, \Gamma +Y_\kappa}( E_q)
 $,
 
\[
\lct( Y_\kappa,\Gamma_\kappa;\fb_{q,\bullet} )  
 =
  \lct( Y,\Gamma +Y_\kappa;\fa_\bullet( \ord_{E_q}  ) )
   \]
   by inversion of adjunction, and $(1+O( \frac{1}{q^2} ))^{1/(n+1)} =1+O(\frac{1}{q^2})$, it follows that
   \[
1 - O\left( \frac{1}{q^2} \right)
\leq 
\frac{ 
\lct( Y,\Gamma +Y_\kappa;
\fa_\bullet( \ord_{E_q} ) )}{A_{Y, \Gamma +Y_\kappa}( E_q)}
\leq 1.\]
Recall, $A_{Y,\Gamma+ Y_\kappa}(E_{q})  = A_{Y, \Gamma}( E_q) = q/r + A$ is of order $O(q)$. 
Therefore, 
\[
A_{Y,\Gamma+ Y_\kappa}(E_{q})
- 
\lct( Y,\Gamma +Y_\kappa;
\fa_{\bullet}(\ord_{E_q} ) )
= 
A_{Y,\Gamma+ Y_\kappa}(E_{q})
\left( 
1 - 
\frac{ \lct( Y,\Gamma +Y_\kappa;
\fa_{\bullet}(\ord_{E_q} ) )}
{A_{Y,\Gamma+ Y_\kappa}(E_{q}) }
 \right)
  \]
 is of order $O(1/q)$ and the desired limit is 0.
 \medskip

\noindent \emph{Claim 2}: For $q\gg0$, there exists an extraction $E_q \subset  Y_q \overset{\mu}{\to} Y$ such that  
\[
(Y_q, \mu_*^{-1} ( \Gamma +Y_\kappa) + E_q) \]
is lc. (By an extraction, we mean $\mu$ is a proper birational morphism, $Y_q$ is normal, $E_q$ appears as a divisor on $Y_q$, and $-E_q$ is $\mu$-ample.)
\medskip

Set  $\e_k : = A_{Y, \Gamma+ Y_\kappa}( E_q )  - \lct(Y, \Gamma+ Y_\kappa; \fa_\bullet( \ord_{E_q})  )$.
Since  $\lim_{q\to \infty} \e_q= 0$ by Claim 1, we may fix $q\gg0$ so that $\e_q <1$. 
Hence, \cite[Prop. 2.2]{BX18} may be applied to get an extraction
$\mu:Y_{q} \to Y$ of $E_q$  with 
\[
(Y_q, \mu_*^{-1} (\Gamma +Y_\kappa) + (1-\e_q) E_q)
\]
 lc.
Since $\lim_{q\to \infty} \e_q= 0$, the ACC for log canonical thresholds \cite{HMX14} implies 
$(Y_q, \mu_*^{-1} (\Gamma +Y_{\kappa}) + E_q)$ is lc for $q \gg0$ and the proof of the claim is complete. 
\medskip

Since $-E_q$ is $\mu$-ample,
$\bigoplus_{p \in \N} \mu_* \cO_{Y_q} (- pE_q)$
is a finitely generated $\cO_Y$-algebra.
Using that 
\[\mu_* \cO_{Y_k} (- pE_q) = \fa_p(\ord_{E_q})  = \bigoplus_{m } \cF^{p-mq} V_m,\]
we see 
\[
\bigoplus_{p \in \N } \bigoplus_{m \in \N }  \cF^{p-mq} V_m = \bigoplus_{m \in \N} \bigoplus_{p \geq -mq} \cF^p V_m
\]
is a finitely generated $V$-algebra. Since $V$ is a finitely generated $R$-algebra, it follows that   
$\underset{m}{\bigoplus} \underset{p}{\bigoplus}(\cF^p V_m)x^{-p}$ is finitely generated $R[x]$-algebra
 and we may  consider the degeneration $(\tilde{\cX},\tilde{\cD}) \to \mathbb{A}^1_R$ by taking ${\rm Proj}$. \medskip

We also consider the degeneration  $(\cY,\tilde{\Gamma})$ of $(Y,\Gamma)$ defined by
\[
\cY : =  \underline{\Spec}_{\mathbb{A}^1_R} \Big( \bigoplus_{p \in \Z}  \fa_p x^{-p} \Big) ,\quad
\text{ where } 
\fa_p: = \mu_{*} \cO_{Y_{q}}(-p E_{q}) \subseteq \cO_{Y}
\]
and $\tilde{\Gamma}$ is the degeneration of $\Gamma$ as in \cite[Defn. 2.19]{LWX18}. 
Since $(Y_q, \mu_*^{-1} ( \Gamma +Y_\kappa) + E_q)$ is lc, a relative version of \cite[Lem. 2.20]{LWX18} implies 
$(\cY, \tilde{\Gamma} + \cY_{x=0} + \cY_\kappa)$ is lc.
Using that $(\tilde{\cX}, \tilde{\cD})$ is a $\mathbb{G}_m$-quotient of an open set of $(\cY, \tilde{\Gamma})$, 
we see $( \tilde{\cX}, \tilde{\cD}+ \tilde{\cX}_\kappa+\tilde{\cX}_{x=0})$ is lc as well.

Observe that $(\tilde{\cX}_K,\tilde{\cD}_K) \to \mathbb{A}^1_K$ is the test configuration induced by the filtration $\cF_K$. 
By \cite[Section 3.2]{Fuj17a}, this test configuration is normal and its Futaki invariant is a multiple of $\beta_{X_K,D_K}(E_K)$,  which is zero. 
Since $\Fut ( \tilde{\cX}_{ \overline{K}},\tilde{\cD}_{\overline{K}})= \Fut(\tilde{\cX}_K,\tilde{\cD}_K)$ and the latter is zero, $( \tilde{\cX}_{ \overline{K}},\tilde{\cD}_{\overline{K}}) \to \mathbb{A}^1_{\overline{K}}$ must be special (otherwise \cite[Thm. 1]{LX14}
would imply there exists a test configuration of $(X_{\overline{K}},D_{\overline{K}})$ with negative Futaki invariant).
Appling Lemma \ref{l-sequivalent} gives that the geometric fiber over $0_K$ is K-semistable.

We will proceed to show
 $(\tilde{\cX},\tilde{\cD}) \to \mathbb{A}^1_R$  is $\Q$-Gorenstein family of log Fano pairs. 
 Since the statement holds over $\{ \pi \neq 0 \}$ and $\{x\neq 0\}$, it remains to consider the behavior over $0 \in \mathbb{A}^1_R$. 
Since $(\tilde{\cX},\tilde{\cD}+ \tilde{\cX}_\kappa+\tilde{\cX}_{x=0})$ is lc, 
$K_{\tilde{\cX}} + \tilde{\cD}$ is $\Q$-Cartier 
and, by \cite[Prop. 2.32.2]{Kol13}, $\tilde{\cD}$ does not contain an irreducible component of $\tilde{\cX}_\kappa \cap \tilde{\cX}_{x=0}$. 
We  are left to show that the geometric fiber over $0$ is a log Fano pair.

First, we claim  that $\tilde{\cX}_\kappa$ is normal.
By Serre's criterion, it suffices to show that $\tilde{\cX}_\kappa$ is $S_2$ and $R_1$.
To verify the first condition, note that  $(\tilde{\cX},\tilde{\cD})$  is klt,  since $(\tilde{\cX},\tilde{\cD}+ \tilde{\cX}_\kappa+\tilde{\cX}_{x=0})$ is lc and klt away from $\{ \pi x = 0 \}$.
Therefore,  \cite[Prop. 5.25]{KM98} implies $\cX_\kappa$ is CM and, hence, $S_2$.
For the second condition, note that $ \Supp(\tilde{\cX}_\kappa+\tilde{\cX}_{x=0})$  has at worst nodes at codimension two points of $\tilde{\cX}$ by  \cite[Prop 2.32.2]{Kol13}. Therefore, $\cX_\kappa$ is $R_1$ in a neighborhood of $\cX_{\kappa} \cap \cX_{x=0}$. Since $\cX_{\kappa}  \setminus \cX_{x=0}\simeq X_\kappa \times (\mathbb{A}^1\setminus 0) $ and $X_\kappa$ is $R_1$, this implies $\cX_\kappa$ is $R_1$.%\footnote{ \HB{I edit the previous two paragraphs and added details on why $\tilde{\cX}_\kappa$ is normal. What I wrote previously missed some details (since I didn't explain why $\cX_\kappa$ was deminormal before applying adjunction. Feel free to edit/shorten (maybe, there is a citation that could shorten this).}}

Now, recall that the Futaki invariant may be written as a combination of intersection numbers of line bundles and intersections number are locally constant in flat projective families.
Therefore,
$\Fut( \tilde{\cX}_{\overline{K} }, \tilde{\cD}_{\overline{K}})=  \Fut( \tilde{\cX}_{\overline{\kappa}} , \tilde{\cD}_{\overline{\kappa}})$ and the latter is also zero.
Since $(X_{\overline{\kappa}}, D_{\overline{\kappa}})$ is K-semistable 
and
 $ \Fut( \tilde{\cX}_{\overline{\kappa}} , \tilde{\cD}_{\overline{\kappa}})=0$,
the test configuration $(\tilde{\cX}_{ \overline{\kappa}},\tilde{\cD}_{\overline{\kappa}}) \to \mathbb{A}^1_{\overline{\kappa}}$ 
must be special (otherwise \cite[Theorem 1]{LX14} would imply there exists a test configuration with negative Futaki invariant). 
Applying Lemma \ref{l-sequivalent} gives that the fiber over $0$ is K-semistable. 
    \end{proof}

\subsubsection{Proof of $\Theta$-reductivity result}
 We will now deduce Theorem \ref{t-Thetared} from Proposition \ref{p-extendE}.

\begin{proof}[Proof of Theorem \ref{t-Thetared}]
Following Section \ref{ss-tpolar}, the test configuration $(\cX_K, \cD_K) \to \mathbb{A}^1_K$, corresponds to a 
$\Z$-filtration $\cG_K$ of $V_K$.
By setting
 \[
 \cG^p V_m : = \cG_K^p V_{K,m} \cap V_m
 \quad \quad \text{ for each $p\in \Z$ },\]
 we get a filtration $\cG$
  of $V_m$ by subbundles, which restricts to the filtration $\cG_K$ over $\Spec(K)$.
 We consider the graded $R[x]$-algebra
$\bigoplus_{m\in \N} \bigoplus_{p \in \Z} \left( \cG^p V_m \right) x^{-p} $
. 
 
Since the test configuration $(\cX_K, \cD_K)$ is special, 
the filtration $\cG_K$ is induced by a divisorial valuation of the form $b\cdot \ord_{E_k}$ over $X_K$ \cite[Claim 5.4]{Fuj16}. 
Specifically, there is a divisor $E_K$ over $X_K$ and  $b\in \Z_{>0}$ so that
\[
\cG_{K}^{p}  V_{K,m}  = \cF_{K}^{   mr A +   \lceil p/b \rceil  } V_{K,m},\]
where $\cF_{K}$ is the filtration of $V_K$ defined by $E_K$ and $A:=A_{X_K,D_K}(E_K)$. 

Observe that $\Fut( \cX_K, \cD_K)=0$.
Indeed, $\Fut(\cX_K, \cD_K)=\Fut (\cX_{\overline{K}}, \cD_{ \overline{K}})$
and the latter is the same as the Futaki invariant associated to the  $\mathbb{G}_m$-action on $(\cX_{\overline{K}}, \cD_{ \overline{K}})_0$. 
Since the Futaki invariants associated to a $\mathbb{G}_m$-action and its inverse add to zero \cite[Lem. 2.23]{LWX18} and $(\cX_{\overline{K}}, \cD_{ \overline{K}})_0$ is K-semistable, they must both be zero.

Now, $\beta_{X_K,D_K}(E_K)$ is a multiple of
 $\Fut(\cX_K,\cD_K)$  by \cite[Thm. 5.1]{Fuj16}. 
 Therefore, the value is zero and we may apply Proposition \ref{p-extendE}.1
to see that ${\bigoplus}_m {\bigoplus}_p \left( \cG^p V_m\right)x^{-p}$ is a finitely generated $R[x]$-module. 
Furthermore, if we set 
$$\tilde{\cX}: =\underline{\Proj}_{\mathbb{A}^1_R} \Big( \bigoplus_{m \in \N} \bigoplus_{p \in \N} \left( \cG^p V_m\right)x^{-p}  \Big)$$ and  $\tilde{\cD}$ equal to the component-wise closure of $D\times  (\mathbb{A}^1 \setminus 0 ) $ under the embedding 
$${X\times  (\mathbb{A}^1 \setminus 0) \simeq\tilde{ \cX}|_{x\neq 0} \hookrightarrow \tilde{\cX}} ,$$ 
then $(\tilde{\cX},\tilde{\cD}) \to \mathbb{A}^1_R$ is  a finite base change of the family considered in Proposition \ref{p-extendE}. Hence, $(\tilde{\cX},\tilde{\cD}) \to \mathbb{A}^1_R$ is a $\Q$-Gorenstein family of log Fano pairs and the geometric fiber over $0\in \mathbb{A}^1_R$  is  K-semistable.
By Lemma \ref{L:filling_conditions}, this is the unique extension of $(\cX,\cD) \to \mathbb{A}^1_R \setminus 0$.%\footnote{\HB{added.}}
\end{proof}

\begin{proof}[Proof of Theorem \ref{t-main}]
The S-completeness and $\Theta$-reductivity of $\kkX^{\rm Kss}_{n,V}$ statements follow immediately from Theorems \ref{t-scomplete} and  \ref{t-Thetared}.
\end{proof}

\begin{proof}[Proof of Corollary \ref{c-goodmoduli}]
It follows from Theorem \ref{t-main} and Lemma \ref{l-substack} that $\kkX$ is S-complete and $\Theta$-reductive with respect to essentially of finite type DVRs.   Theorem \ref{T:existence} and Remark \ref{R:existence} imply 
that $\kkX$ has a separated good moduli space. 
\end{proof}

\bigskip

\end{document}